\documentclass[11pt,reqno]{amsproc}
\allowdisplaybreaks

\usepackage{amssymb}
\usepackage{stmaryrd}

\usepackage{enumerate}
\usepackage{fullpage}
\usepackage{textcomp}
\usepackage{lettrine}

\usepackage{graphicx}
\usepackage{subfigure}

\usepackage[font=small,labelfont=bf,
   justification=justified,
   format=plain]{caption} 

\usepackage[debug=false, colorlinks=true, pdfstartview=FitV,
linkcolor=blue, citecolor=blue, urlcolor=blue]{hyperref}
\usepackage[semicolon,square,authoryear,sort]{natbib}

\usepackage[doipre={DOI:~}]{uri}

\usepackage{color}
\usepackage{colortbl}
\usepackage[usenames,dvipsnames,table]{xcolor}
\usepackage[most]{tcolorbox}

\newlength{\drop}
\definecolor{amethyst}{rgb}{0.6, 0.4, 0.8}
\definecolor{burgundy}{rgb}{0.5, 0.0, 0.13}

\usepackage{longtable} 
\usepackage{multirow}
\usepackage{rotating} 
\usepackage{bigstrut} 
\usepackage{hhline} 

\usepackage{amsthm}

\newtheoremstyle{remboldstyle}
  {}{}{}{}{\bfseries}{.}{.5em}{{\thmname{#1 }}{\thmnumber{#2}}{\thmnote{ (#3)}}}
\theoremstyle{remboldstyle}

\newtheorem{theorem}{Theorem}[section]

\newtheorem{remark}{Remark}[section]

\title{\textbf{Effect of temperature-dependent material properties on thermal regulation in microvascular composites}}

\author{\textbf{K.~Adhikari$^{1}$}, \textbf{J.~F.~Patrick$^{2}$}, and \textbf{K.~B.~Nakshatrala$^{1}$} \\
  {\small 
  $^{1}$ Department of Civil and Environmental Engineering \\
  University of Houston, Houston, TX 77204, USA. \\
  $^{2}$ Department of Civil, Construction, and Environmental Engineering \\ 
  North Carolina State University, 
  Raleigh, NC 27695, USA. 
  } \\
  \textbf{Correspondence to:} knakshatrala@uh.edu, 
  +1-713-743-4418
  }


\begin{document}

\begin{titlepage}
  \drop=0.1\textheight
  \centering
  \vspace*{0.1\baselineskip}
  \rule{\textwidth}{1.6pt}\vspace*{-\baselineskip}\vspace*{1pt}
  \rule{\textwidth}{0.4pt}\\[\baselineskip]
       {\Large \textbf{\color{burgundy}
       Effect of temperature-dependent material properties\\[0.1\baselineskip]on thermal regulation in microvascular composites}}\\[0.1\baselineskip]
       \rule{\textwidth}{0.4pt}\vspace*{-\baselineskip}\vspace{1pt}
       \rule{\textwidth}{1.6pt}\\[0.1\baselineskip]
       \scshape
       An e-print of this paper is available on arXiv. \par
       \vspace*{0.1\baselineskip}
       Authored by \\[0.1\baselineskip]

{\Large K.~Adhikari\par}
  {\itshape Graduate Student, Department of Civil \& Environmental Engineering \\
  University of Houston, Houston, Texas 77204.}\\[0.1\baselineskip]

  {\Large J.~F.~Patrick\par}
  {\itshape Assistant Professor, Department of Civil, Construction, \& Environmental Engineering \\
  North Carolina State University, Raleigh, North Carolina 27695.}\\[0.1\baselineskip]
  
  {\Large K.~B.~Nakshatrala\par}
  {\itshape Department of Civil \& Environmental Engineering \\
  University of Houston, Houston, Texas 77204. \\
  \textbf{phone:} +1-713-743-4418, \textbf{e-mail:} knakshatrala@uh.edu \\
  \textbf{website:} http://www.cive.uh.edu/faculty/nakshatrala}\\[0.0\baselineskip]
\begin{figure*}[ht]
    \centering
    \includegraphics[scale=0.4]{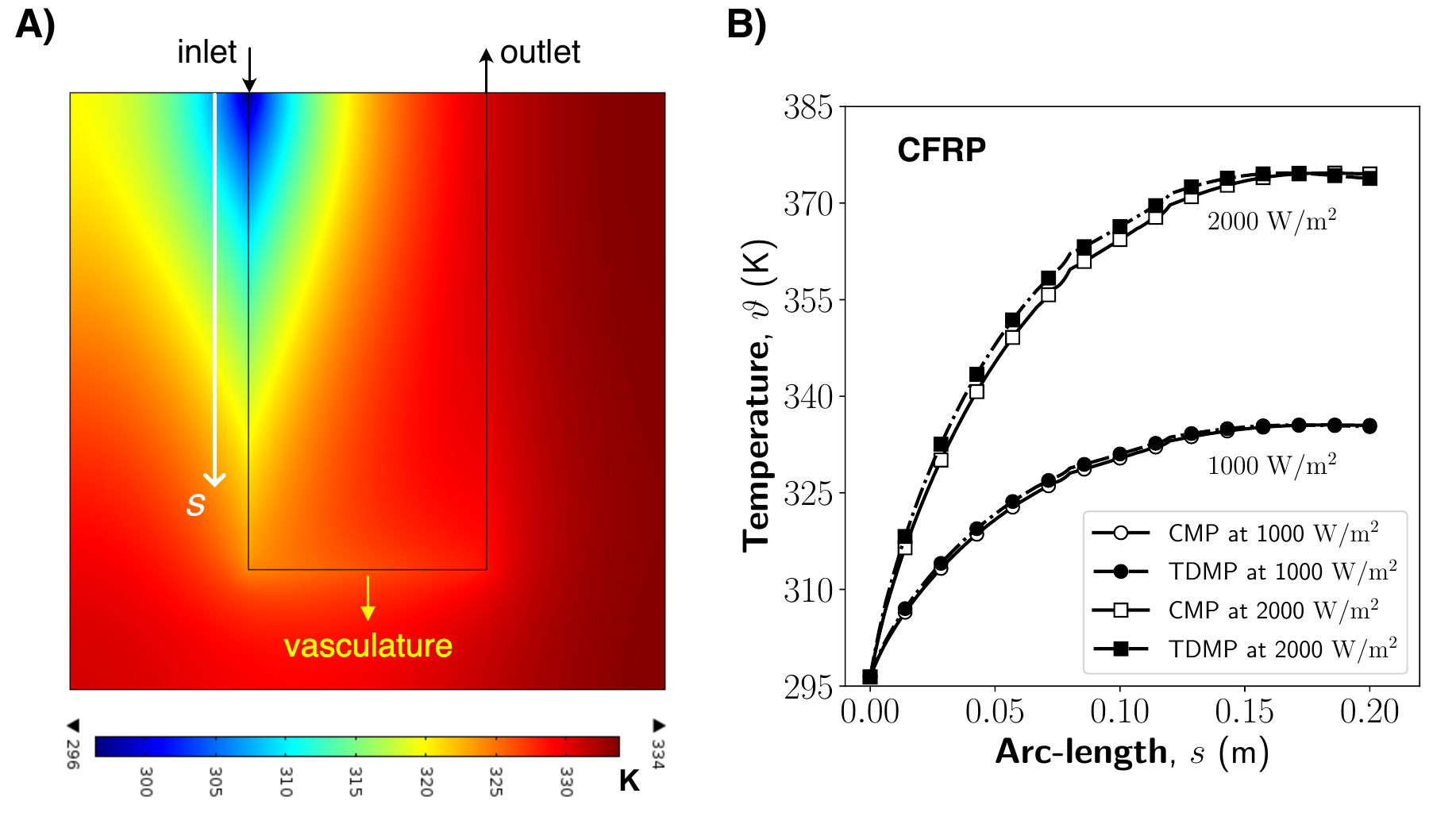}
    \vspace{-0.05in}   
    \captionsetup{labelformat=empty}
    \caption{This figure depicts \textbf{A)} the temperature field in the entire domain and \textbf{B)} the variation of the temperature ($\vartheta$) along the arc-length ($s$) of the vasculature. The left image shows the U-shaped vasculature in a carbon-fiber-reinforced polymer (CFRP) host whose thermophysical properties depend on the temperature. \emph{Inference:} For two different applied heat fluxes, $1000$ and $2000 \; \mathrm{W/m^2}$, numerical results reveal that the temperature-dependent material properties (TDMP) did not significantly influence the solution field compared to assumed constant material properties (CMP).}
\end{figure*}
\vspace{-0.1in}
  {\scshape 2024} \\
  {\small Computational \& Applied Mechanics Laboratory} \par
\end{titlepage}

\begin{abstract}
Fiber-reinforced composites (FRC) provide structural systems with unique features that appeal to various civilian and military sectors. Often, one needs to modulate the temperature field to achieve the intended functionalities (e.g., self-healing) in these lightweight structures. Vascular-based active cooling offers one efficient way of thermal regulation in such material systems. However, the thermophysical properties (e.g., thermal conductivity, specific heat capacity) of FRC and their base constituents depend on temperature, and such structures are often subject to a broad spectrum of temperatures. Notably, prior active cooling modeling studies did not account for such temperature dependence. Thus, the primary aim of this paper is to reveal the effect of temperature-dependent material properties---obtained via material characterization---on the qualitative and quantitative behaviors of active cooling. By applying mathematical analysis and conducting numerical simulations, we show this dependence does not affect qualitative attributes, such as minimum and maximum principles (in the same spirit as \textsc{Eberhard Hopf}'s results for elliptic partial differential equations). However, the dependence slightly affects quantitative results, such as the mean surface temperature and thermal efficiency. The import of our study is that it provides a deeper understanding of thermal regulation systems under practical scenarios and can guide researchers and practitioners in perfecting associated designs. 
\end{abstract}

\maketitle

\vspace{-0.3in}

\setcounter{figure}{0}  

\section*{A LIST OF ABBREVIATIONS}
\setcounter{table}{-1}
\begin{longtable}{ll}\hline
    \hline \multicolumn{2}{c}{\emph{Abbreviations}} \\ \hline
  \textsf{CFRP} & carbon-fiber-reinforced polymer \\ 
  \textsf{CMP} & constant material properties \\  
  \textsf{FRC} & fiber-reinforced composites \\  
  \textsf{GFRP} & glass-fiber-reinforced polymer \\ 
  \textsf{IBVP} & initial boundary value problem\\
  \textsf{MST} & mean surface temperature \\
  \textsf{ROM} & reduced-order model \\
  \textsf{TDMP} & temperature-dependent material properties \\
  \hline
\end{longtable}

\setcounter{table}{0}

\newpage 

\section{INTRODUCTION AND MOTIVATION}
\label{Sec:S1_TProp_Intro}

\lettrine[findent=2pt]{\fbox{\textbf{T}}}{hermophysical properties of materials often depend on temperature}. Therefore, understanding how materials respond to changes in temperature is crucial for designing efficient systems, predicting material performance, and ensuring the safety and reliability of processes and products. This dependence is particularly conspicuous for fiber-reinforced composites (FRC) and their base constituents \citep{yu2007modeling}. FRCs consist of two main components: a continuous matrix, often composed of polymers or resins, which contributes to structural integrity and cohesion, and fiber reinforcements, typically made of various materials such as glass or carbon, which provide stiffness and strength \citep{hyer2009stress,ashbee1993fundamental}. Two properties pertinent to \emph{thermal regulation in micro-vascular composites}---the main focus of this paper---are the specific heat capacity and thermal conductivity. 

Specific heat capacity is the amount of energy required---supplied as heat---to raise the temperature of one kilogram of the material by one Kelvin, and this needed energy can vary with temperature \citep{kittel1998thermal}. On the other hand, thermal conductivity measures a substance's propensity to transfer heat; it is the rate at which thermal energy flows through a unit area for a unit temperature difference across the material \citep{cengel2011thermodynamics}. This propensity---consequently the thermal conductivity---can also vary with temperature  \citep{tritt2005thermal}. The underlying physical mechanism is that a temperature change alters either the kinetic or potential energy of the atoms within the material, subsequently leading to a change in material properties \citep{grimvall1999thermophysical,buck2011thermal}. Studies conducted on a range of materials---including ceramics \citep{de2012review}, VO$_2$ \citep{ordonez2018modeling}, coal and rocks \citep{wen2015temperature}, and polymers \citep{godovsky2012thermophysical}---have reported changes in material properties (such as thermal conductivity, diffusivity, and specific heat capacity) due to change in temperature. As a result, parameter estimation techniques based on inverse methods have been developed to estimate temperature-dependent thermal properties \citep{cui2012new,huang1995inverse}.

This paper primarily focuses on understanding the effect of temperature-dependent thermophysical properties on thermal regulation in microvascular polymer-matrix composites. To understand the importance of our work, it is worth looking at how thermal regulation is essential for optimal performance and long-term sustainability in numerous synthetic and natural systems. Thermal regulation is inherent to mammals and birds, enabling them to adapt to temperature variations \citep{wissler2018animal}. The legs of birds serve as controlled heat conduits that play an important role in thermoregulation \citep{steen1965importance, kahl1963thermoregulation}. Thermal regulation based on fluid flow in an embedded vasculature (i.e., active cooling) is also central to how jackrabbits adapt to desert climate \citep{hill1976jackrabbit} and how the human body maintains homeostasis \citep{gonzalez2012human}, to name a few. 

In numerous contemporary technologies, the combination of compact dimensions and enhanced operational capabilities constitutes a major challenge in addressing device overheating \citep{singh2006thermal}. For instance, state-of-the-art high-energy-density lithium-ion batteries produce more heat than conventional batteries, and the performance of these batteries and the structural integrity diminish at higher temperatures \citep{nitta2015li,bandhauer2011critical}. Thus, to make these emerging technologies successful, designers integrate thermal regulation systems into product development to achieve high efficiency while minimizing energy consumption. Several possible innovative techniques exist for thermal regulation in synthetic materials, for example, thermal control coatings, thermal metamaterials, thermal switches, and reciprocating airflow \citep{kiomarsipour2013improvement,sklan2018thermal,yang2019integrated,mahamud2011reciprocating}. 

However, bio-inspired thermal regulation offers attractive and cost-effective solutions for synthetic material systems \citep{swanson2003nasa,Nakshatrala_PNAS_Nexus_2023,devi2021microvascular}. This approach entails the circulation of a coolant through a network to control temperature, emulating the circulatory system found in living organisms. Over the past few decades, researchers have made significant advancements in bringing fluid-induced vascular-based thermal regulation to numerous applications: nano-satellites \citep{tan2018computational}, batteries \citep{pety2017carbon}, microelectronics \citep{spencer2007air}, and space probes \citep{driesman2019journey}. Within these vascular-based systems, a network of micro-channels is incorporated into the synthetic material utilizing the latest progress in additive manufacturing techniques \citep{patrick2017robust,nguyen2018recent,favero2021additive,pathak4529173two}. The flow of a coolant through vasculature has a dual effect on the temperature of the surrounding material. Firstly, it extracts heat directly from the thermally loaded material \citep{oueslati2008pcb,pastukhov2003miniature}. Secondly, it facilitates the redistribution of heat between hotter and colder regions within the domain \citep{benfeldt1999vivo}. 

Given the complexity of thermal regulation, modeling plays a crucial role, and duly, various modeling approaches have been previously employed \citep{jagtap2023coolpinns,nakshatrala2023ROM}. However, prior modeling efforts on vascular-based thermal regulation did not account for temperature-dependent material properties. Three prior studies that are closely connected to this paper need to be addressed when introducing temperature-dependent material properties, as they previously assumed constant values sampled at ambient temperature.

First, \citet{nakshatrala2023ROM} has presented a mathematical model to describe the steady-state response of a vascular-based thermal regulation system; the cited paper has also established qualitative mathematical properties---such as the minimum and maximum principles---that the steady-state solutions satisfy. Thus, a mathematical-related question is:  
\begin{enumerate}[(Q1)]
    \item Do the minimum and maximum principles hold for thermal regulation even under temperature-dependent material properties?  
\end{enumerate}

Second, \citet{Nakshatrala_PNAS_Nexus_2023} have recently shown that the mean surface temperature (MST) and the outlet temperature remain invariant under flow reversal (i.e., swapping the locations of the inlet and outlet), assuming the material properties were independent of temperature. Duly, a natural question to ask is: 
\begin{enumerate}[(Q2)]
    \item How do temperature-dependent material properties affect the two invariants (i.e., MST and outlet temperature) under flow reversal? 
\end{enumerate}

Third, \citet{nakshatrala2023thermal} have performed an adjoint-state-based sensitivity analysis to unravel the non-monotonic behavior of the host material's thermal conductivity to the MST without considering the temperature-dependent material properties. Thus, besides studying qualitative properties, it is also necessary to undertake a quantitative investigation. The third and final question is: 
\begin{enumerate}[(Q3)]
    \item How does accounting for temperature-dependent material properties affect the active cooling performance quantitatively?
\end{enumerate}

We answer these three posed questions in subsequent sections of this paper. We use mathematical analysis and numerical simulations to discern the effect of temperature-dependent material properties on active cooling performance. The significance of our work is that it provides a deeper understanding of how vascular-based thermal regulation systems behave in realistic scenarios.  

The layout for the rest of this paper is as follows. We introduce a mathematical model that describes thermal regulation in thin vascular systems, considering temperature-dependent material properties (\S\ref{Sec:S2_TProp_ROM}). Next, we present thermal characterization results, showing how the specific heat capacity and thermal conductivity vary with temperature for CFRP, GFRP, and epoxy (\S\ref{Sec:S3_TProp_MatProperties}). Subsequently, we establish minimum and maximum principles for steady-state solutions, extending the previously known results for constant material properties to the case of temperature-dependent material properties---answering the first question (\S\ref{Sec:S4_TProp_MP}). Using numerical simulations, we then address the last two questions (\S\ref{Sec:S5_TProp_NR}--\S\ref{Sec:S7_TProp_Quantitative}). Finally, we draw concluding remarks and propose potential future work (\S\ref{Sec:S8_TProp_Closure}).

\section{TEMPERATURE-DEPENDENT REDUCED-ORDER MODEL}
\label{Sec:S2_TProp_ROM}

\begin{figure}[h]
    \centering
    \includegraphics[scale=0.5]{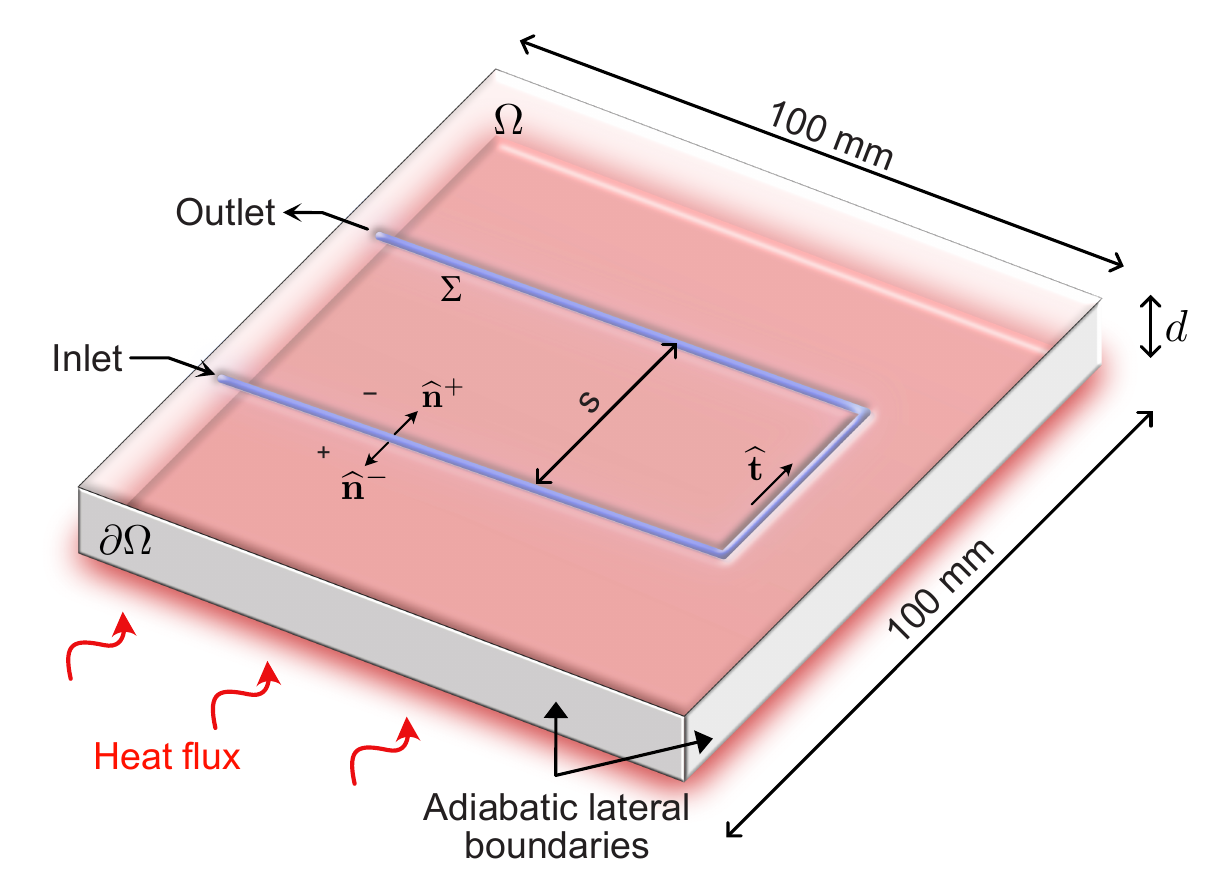}
    \caption{\textsf{Problem setup:} The figure illustrates a thin domain $\Omega$ of area 100 mm x 100 mm and thickness $d$. The lateral boundaries $\partial\Omega$ are adiabatic. A U-shaped vasculature $\Sigma$ with spacing $s$ is embedded with an inlet and outlet through which coolant flows. $\widehat{\mathbf{n}}^{\pm}(\mathbf{x})$ denote the unit normals on either side of $\Sigma$ and the tangent vector along the vasculature is denoted by $\widehat{\mathbf{t}}(\mathbf{x})$. A constant heat flux is applied on the bottom face of the domain, and the top surface is free to radiate and convect.}
    \label{fig1:TProp_Problem_setup}
\end{figure}

Consider a slender body with thickness $d$, as shown in \textbf{Fig.~ \ref{fig1:TProp_Problem_setup}}. The body contains an embedded vasculature---a network of channels that carry coolant to modulate the temperature field. We treat the embedded vasculature as a curve, meaning that we do not resolve the flow characteristics across the cross-section of the vasculature. The slenderness assumption, along with treating the vasculature as a curve, renders a complete three-dimensional analysis unnecessary for modeling thermal regulation in thin vascular systems. Rather a two-dimensional description will be adequate. Accordingly, we utilize the mathematical model proposed by \citet{nakshatrala2023ROM} and modify it to incorporate temperature-dependent material properties. Henceforth, we refer to the resulting mathematical description as the reduced-order model (ROM).

Notationally, we denote the two-dimensional mid-surface---referred to as the domain---$\Omega$, and the domain's external boundary $\partial \Omega:= \overline{\Omega}\setminus\Omega$, where an overline denotes the set closure. We use $\mathbf{x} \in \overline{\Omega}$ to depict a spatial point, and $t \in [0,\mathcal{T}]$ is the time, with $\mathcal{T}$ denoting the length of the time interval of interest. $\partial(\cdot)/\partial t$ stands for the partial derivative in time, and $\mathrm{grad}[\cdot]$ and $\mathrm{div}[\cdot]$ represent the spatial gradient and divergence operators, respectively. $\vartheta(\mathbf{x},t)$ and $\mathbf{q}(\mathbf{x},t)$ are, respectively, the temperature and heat flux vector fields defined at a spatial point $\mathbf{x}$ and at time instance $t$.

We assume the external boundary ($\partial \Omega$) is piece-wise smooth with the unit outward normal vector $\widehat{\mathbf{n}}(\mathbf{x})$ defined everywhere on $\partial \Omega$ except at corners. Regarding boundary conditions, the external boundary is decomposed into complementary parts: $\Gamma^{\vartheta}$ and $\Gamma^{q}$. $\Gamma^\vartheta$ denotes the part of the boundary on which the Dirichlet boundary condition is enforced (i.e., temperature is prescribed), and $\Gamma^{q}$ is that part of the boundary on which the Neumann boundary condition is enforced (i.e., heat flux is prescribed). For  mathematical well-posedness, we require that $\Gamma^\vartheta \cup \Gamma^{q} = \partial \Omega$ and $\Gamma^\vartheta \cap \Gamma^{q} = \emptyset$. 

We denote the curve representing the vasculature by $\Sigma$ with $\widehat{\mathbf{t}}(\mathbf{x})$ denoting the unit tangent vector along the vasculature; see \citep[SI]{Nakshatrala_PNAS_Nexus_2023} for a discussion on how to mathematically define $\widehat{\mathbf{t}}(\mathbf{x})$. We parameterize $\Sigma$ using the arc-length $s$ starting from the inlet; patently, $s=0$ corresponds to the inlet. To achieve active cooling, a fluid flows through the vasculature where $\rho_f$ and $c_f$ denote the density and specific heat capacity of this fluid, respectively. The vasculature's inlet and outlet are on the lateral sides of the domain. $\vartheta_{\mathrm{inlet}}$ denotes the prescribed temperature at the inlet, while the outlet temperature is unknown \emph{a priori} and is part of the solution of the mathematical model. We utilize the jump operator to write the balance laws across the vasculature (i.e., jump conditions). Given a scalar field $\alpha(\mathbf{x},t)$ and a vector field $\mathbf{a}(\mathbf{x},t)$, the jump operator is defined as follows: 
\begin{subequations}
    \label{Eqn:TProp_jump_operator}
    \begin{align}
    &\llbracket \alpha(\mathbf{x},t)\rrbracket = \alpha^{+}(\mathbf{x},t) \, \widehat{\mathbf{n}}^{+}(\mathbf{x}) 
    + \alpha^{-}(\mathbf{x},t) \, \widehat{\mathbf{n}}^{-}(\mathbf{x}) \\
    &\llbracket \mathbf{a}(\mathbf{x},t)\rrbracket = \mathbf{a}^{+}(\mathbf{x},t) \bullet \widehat{\mathbf{n}}^{+}(\mathbf{x}) 
    + \mathbf{a}^{-}(\mathbf{x},t) \bullet
    \widehat{\mathbf{n}}^{-}(\mathbf{x})
    \end{align}
\end{subequations}
where $\bullet$ denotes the standard inner product, $\widehat{\mathbf{n}}^{\pm}(\mathbf{x})$ are unit normal vectors on either side of the vasculature, and $\alpha^{\pm}(\mathbf{x},t)$ represent the limiting values of the scalar field on either side of $\Sigma$. A similar definition holds for $\mathbf{a}^{\pm}(\mathbf{x},t)$. For further details on the jump operator, refer to \citep{nakshatrala2023ROM,nakshatrala2023thermal}. 

Besides the heat transported by the flowing fluid (i.e., active cooling), there are three other modes of heat transfer: the bottom surface is subject to applied heat flux, the host solid can conduct heat, and the top surface is free to convect and radiate. $h_{T}$ denotes the heat transfer coefficient, $\varepsilon$ the emissivity, and $\sigma \approx 5.67 \times 10^{-8} \; \mathrm{W}  \mathrm{m}^{-2} \mathrm{K}^{-4}$ the Stefan-Boltzmann constant. It is reasonable to assume that $h_{T}$ and $\varepsilon$, which are not properties of the bulk material, do not depend on temperature, at least for the range considered in this paper \citep{devi2023methodology}. $\rho_{s}$ denotes the density of the host solid, while $\mathbf{K}\big(\mathbf{x},\vartheta(\mathbf{x},t)\big)$ and $c_{s}\big(\vartheta(\mathbf{x},t)\big)$ represent the temperature-dependent thermal conductivity and specific heat capacity of the host solid, respectively. 

Physics demands that $c_{s}\big(\vartheta(\mathbf{x},t)\big)$ is positive and $\mathbf{K}\big(\mathbf{x},\vartheta(\mathbf{x},t)\big)$ is symmetric and positive definite. For mathematical analysis, we, however, require a stronger condition: the thermal conductivity is uniformly elliptic \citep{gilbarg2015elliptic}. That is, there exists a constant $k_1 > 0$ such that:
\begin{align}
    \label{Eqn:TProp_thermal_conductivity_positive_definite}
    0 < k_1 \mathbf{y} \bullet \mathbf{y} 
    \leq \mathbf{y} \bullet \mathbf{K}\big(\mathbf{x},\vartheta(\mathbf{x},t)\big) \, \mathbf{y} 
    \quad \forall \mathbf{y} \setminus\{\mathbf{0}\} \in \mathbb{R}^2
\end{align}
where $\mathbb{R}^{2}$ denotes the two-dimensional Euclidean space. In some sections of this paper (e.g., \S\ref{Sec:S3_TProp_MatProperties} and \S\ref{Sec:S5_TProp_NR}), we assume the thermal conductivity is isotropic:
\begin{align}
    \mathbf{K}\big(\mathbf{x},\vartheta(\mathbf{x},t)\big) 
    = k_{s} \big(\mathbf{x},\vartheta(\mathbf{x},t)\big) \,  \mathbf{I}
\end{align}
where $\mathbf{I}$ represents the second-order identity tensor, and $k_{s}(\mathbf{x},\vartheta)$ is the scalar conductivity field.
 
The governing equations for the temperature-dependent ROM take the following form:
\begin{alignat}{2}
    \label{Eqn:Temp_GE_BoE}
    &d \, \rho_{s} \, c_{s}(\vartheta) \frac{\partial \vartheta}{\partial t} + d \, \mathrm{div}[\mathbf{q}(\mathbf{x},t)] 
    = f(\mathbf{x},t) 
    - h_{T} \, \big(\vartheta(\mathbf{x},t) - \vartheta_{\mathrm{amb}}\big)
    \nonumber \\
    &\hspace{2.5in} -\varepsilon \, \sigma \, \big(\vartheta^{4}(\mathbf{x},t) - \vartheta^4_{\mathrm{amb}}\big) 
    && \quad \forall \mathbf{x} \in \Omega, \forall t \in (0,\mathcal{T}] \\
    \label{Eqn:Temp_GE_Fourier}
    &\mathbf{q}(\mathbf{x},t) 
    = - \mathbf{K}\big(\mathbf{x},\vartheta(\mathbf{x},t)\big) \, \mathrm{grad}[\vartheta(\mathbf{x},t)] 
    &&  \quad \forall \mathbf{x} \in \Omega, \forall t \in (0,\mathcal{T}] \\
    \label{Eqn:Temp_GE_temperature_BC}
    &\vartheta(\mathbf{x},t) = \vartheta^{\mathrm{p}}(\mathbf{x},t) 
    &&  \quad \forall \mathbf{x} \in \Gamma^{\vartheta}, \forall t \in [0,\mathcal{T}] \\
    \label{Eqn:Temp_GE_heat_flux_BC}
    &d \, \mathbf{q}(\mathbf{x},t) \bullet \widehat{\mathbf{n}}(\mathbf{x}) = q^{\mathrm{p}}(\mathbf{x},t) 
    && \quad \forall \mathbf{x} \in \Gamma^{q}, 
    \forall t \in [0,\mathcal{T}] \\
    \label{Eqn:Temp_GE_jump_theta}
    &\llbracket\vartheta(\mathbf{x},t)\rrbracket = \mathbf{0} 
    && \quad \forall \mathbf{x} \in \Sigma, \forall t \in [0,\mathcal{T}] \\ 
    \label{Eqn:Temp_GE_jump_q}
    &d \, \llbracket\mathbf{q}(\mathbf{x},t)\rrbracket = \chi \,  \mathrm{grad}[\vartheta] \bullet \widehat{\mathbf{t}}(\mathbf{x}) 
    && \quad \forall \mathbf{x} \in \Sigma, \forall t \in [0,\mathcal{T}] \\ 
    \label{Eqn:Temp_GE_IC}
    &\vartheta(\mathbf{x},t=0) = \vartheta_{\mathrm{initial}}(\mathbf{x})
    && \quad \forall \mathbf{x} \in \Omega \\
    \label{Eqn:Temp_GE_inlet}
    &\vartheta(\mathbf{x},t) = \vartheta_{\mathrm{inlet}}
    && \quad \mbox{at inlet}, \forall t \in [0,\mathcal{T}]
\end{alignat}
where $f(\mathbf{x},t)$ denotes the applied heat flux on the bottom surface, $\vartheta_{\mathrm{amb}}$ the ambient temperature,  $\vartheta_{\mathrm{initial}}(\mathbf{x})$ the prescribed initial temperature distribution, $\vartheta^{\mathrm{p}}(\mathbf{x},t)$ the prescribed temperature on the boundary, $q^{\mathrm{p}}(\mathbf{x},t)$ the prescribed heat flux on the boundary, and $\chi$ the heat capacity rate, defined as follows:
\begin{align}
    \label{Eqn:TProp_heat_capacity_rate}
    \chi = \rho_{f} \, Q \, c_{f}
\end{align}
with $Q$ denoting the volumetric flow rate.

Equation \eqref{Eqn:Temp_GE_BoE} is the balance of energy, accounting for conduction, convection, radiation, applied heat flux, and increase in internal energy. Equation~\eqref{Eqn:Temp_GE_Fourier} represents the Fourier model, describing the heat conduction. Equations \eqref{Eqn:Temp_GE_temperature_BC} and \eqref{Eqn:Temp_GE_heat_flux_BC} are the boundary conditions, while Eqs.~\eqref{Eqn:Temp_GE_jump_theta} and  \eqref{Eqn:Temp_GE_jump_q} are the jump conditions for the temperature and the heat flux, respectively. The next two equations represent the initial condition and the prescribed temperature at the inlet. The temperature for classical (i.e., non-spin) systems is positive when measured in the Kelvin scale \citep{kittel1998thermal}. This non-negative physical constraint will be crucial in establishing minimum and maximum principles (see \S\ref{Sec:S4_TProp_MP}).

Since the thermal conductivity could depend on the temperature, the above initial boundary value problem (IBVP) is, in its generality, a quasi-linear parabolic differential equation \citep{pao2012nonlinear}.

\section{TEMPERATURE-DEPENDENT MATERIAL PROPERTIES}
\label{Sec:S3_TProp_MatProperties}
As previously stated, the thermophysical properties of CFRP, GFRP, and epoxy---common structural composites and polymer matrix material---are known to depend on the temperature. Therefore, we have utilized the facilities at the \emph{Thermophysical Properties Research Laboratory} to perform thermal characterization of these materials \citep{TPRL}. Specific heat capacity and thermal conductivity were measured in the range from room temperature (RT $\approx$ 296.15 K) to 423.15 K. 
 
A \emph{Perkin-Elmer Differential Scanning Calorimeter} (ASTM E1269), with sapphire as the reference material, was used to measure the specific heat capacity $(c_s)$ of these three materials. As shown in \textbf{Fig.~ \ref{Fig2:TProp_Variation_of_host_properties}A}, the specific heat capacity for all three materials strongly depends on the temperature, with epoxy having the maximum variation. To facilitate numerical simulations and reveal the variation's nature, we provided an interpolation function (i.e., a polynomial fit) to capture the trend for each material.

The thermal diffusivity ($\alpha$) was measured using the laser flash method. The bulk density ($\rho_s$) was estimated directly from the sample's geometry and mass. Then, the thermal conductivity ($k_s$) was calculated as a product of the bulk density, specific heat capacity, and thermal diffusivity (i.e., $k_s = \rho_s \, c_s \, \alpha$). As shown in \textbf{Fig.~\ref{Fig2:TProp_Variation_of_host_properties}B}, the temperature had minimal impact on the thermal conductivity of epoxy and GFRP, whereas CFRP had a notable variation in the chosen temperature range.
\begin{figure}[!h]
    \centering
    \includegraphics[scale=0.44]{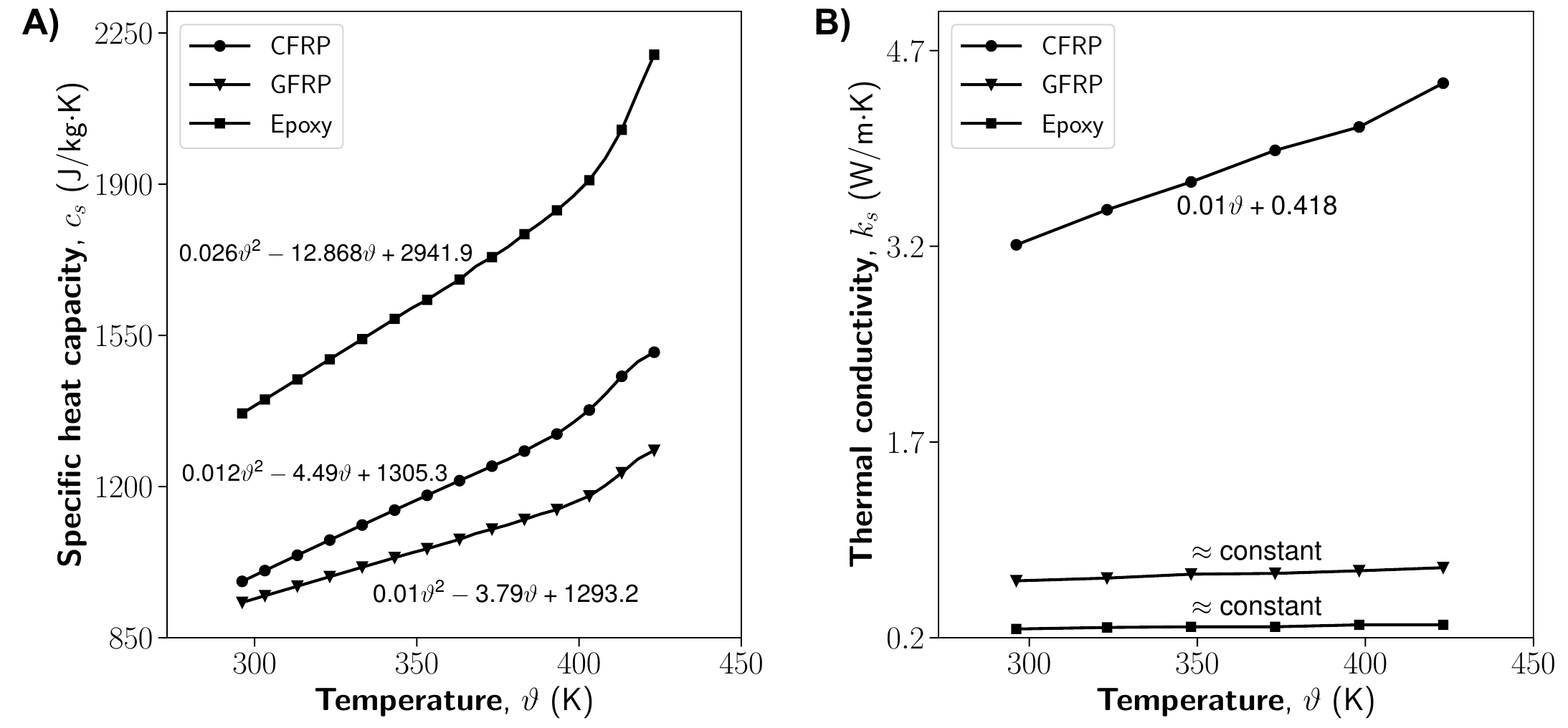}
    \caption{\textsf{Solid host material properties:} This figure shows the variations of the \textbf{A)} specific heat capacity and \textbf{B)} thermal conductivity --- in the temperature range 296.15--423.15 K for three candidate host materials: CFRP, GFRP, and epoxy. The figure also provides the best-fit polynomial for each curve. The specific heat capacity varies significantly with temperature for all three materials. However, the thermal conductivity remains relatively constant except for CFRP.
    \label{Fig2:TProp_Variation_of_host_properties}}
\end{figure}

For water (i.e., the coolant), we took the density and specific heat capacity values for the temperature range 0--100 $^\circ \mathrm{C}$ from \citep{scienceschool} and \citep{engineeringtoolbox}, respectively. As shown in \textbf{Fig.~\ref{Fig3:TProp_Water_properties}}, the heat capacity rate did not change appreciably with temperature for a volumetric flow rate of $Q = 1 \, \mathrm{mL/min}$, which lies in the typical range of flow rates used for active cooling of microvascular composites \citep{devi2021microvascular}. \emph{Therefore, we took the coolant's material properties to be constant in the rest of this paper.}

In the subsequent sections, we investigate how the temperature-dependent properties impact the active cooling performance.

\begin{figure}[!h]
    \centering
    \includegraphics[scale=0.43]{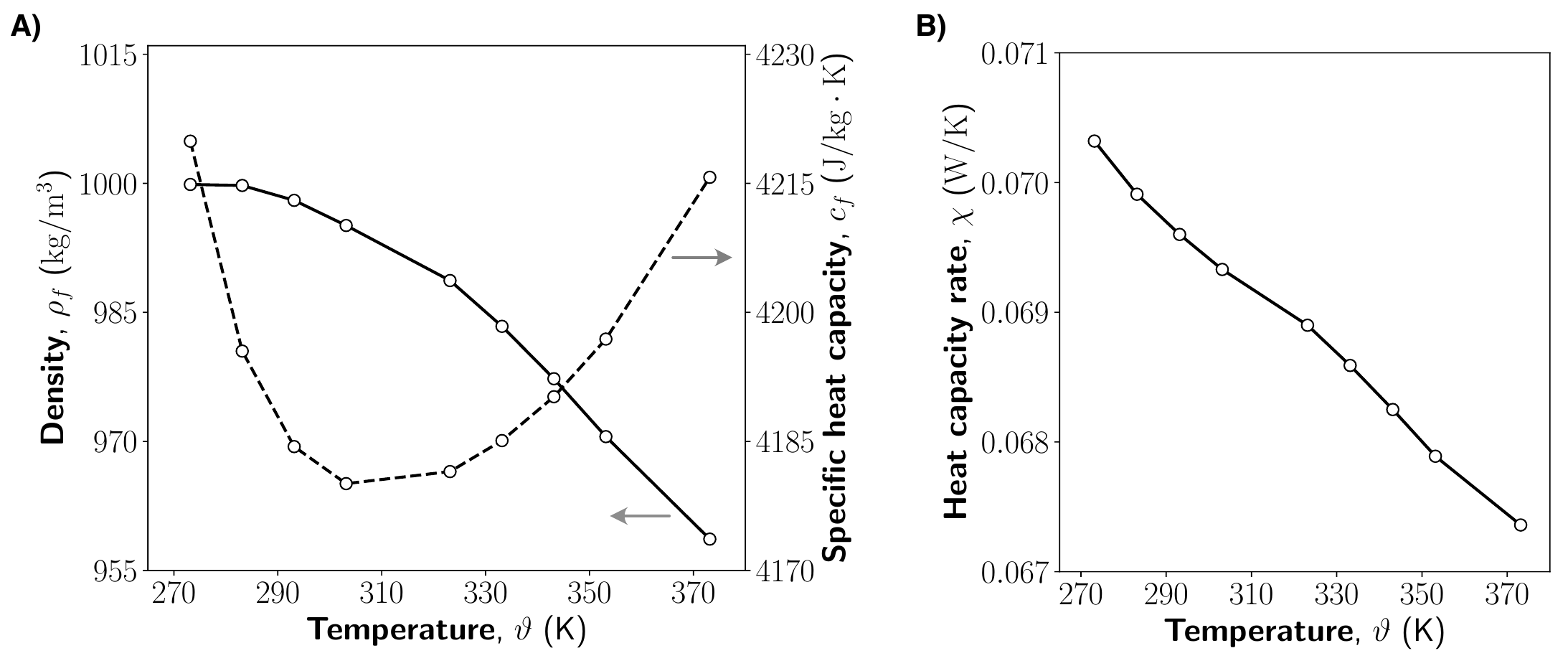}
    \caption{\textsf{Liquid coolant (fluid) properties:} \textbf{A)} The density and specific heat capacity of water (coolant) vary only 4.30\% and 0.95\%, respectively, over the temperature range 0--100 $^\circ\mathrm{C}$. These data were sourced from \citep{scienceschool} and \citep{engineeringtoolbox}, respectively. \textbf{B)} The heat capacity rate (i.e., the product of the density, specific heat capacity, and the volumetric flow rate) also varies merely 4.39\% over the chosen temperature range for the volumetric flow rate $Q = 1 \, \mathrm{mL/min}$. Hence, in this paper, we have taken the fluid's heat capacity, which is an input to the mathematical model, to be constant.     \label{Fig3:TProp_Water_properties}}
\end{figure}

\pagebreak
\section{MINIMUM AND MAXIMUM PRINCIPLES}
\label{Sec:S4_TProp_MP}
Recently, \citet{nakshatrala2023ROM} has established minimum and maximum principles under steady-state regime for \emph{constant} material properties. This section shows that such minimum and maximum principles hold even when the material properties depend on the temperature.  

In the steady-state regime, the governing equations of the mathematical model take the following form:
\begin{subequations}
\begin{alignat}{2}
    \label{Eqn:Temp_SS_GE_BoE}
    -&d \, \mathrm{div}\big[\mathbf{K}\big(\mathbf{x},\vartheta(\mathbf{x})\big) \, \mathrm{grad}[\vartheta(\mathbf{x})] \big] 
    = f(\mathbf{x}) 
    - h_{T} \, \big(\vartheta(\mathbf{x}) - \vartheta_{\mathrm{amb}}\big)
    -\varepsilon \, \sigma \, \big(\vartheta^{4}(\mathbf{x}) - \vartheta^4_{\mathrm{amb}}\big) 
    && \quad \forall \mathbf{x} \in \Omega \\
    \label{Eqn:Temp_SS_GE_temperature_BC}
    &\vartheta(\mathbf{x}) = \vartheta^{\mathrm{p}}(\mathbf{x}) 
    &&  \quad \forall \mathbf{x} \in \Gamma^{\vartheta} \\
    \label{Eqn:Temp_SS_GE_heat_flux_BC}
    -&d \, \widehat{\mathbf{n}}(\mathbf{x}) \bullet \mathbf{K}\big(\mathbf{x},\vartheta(\mathbf{x})\big) \, \mathrm{grad}[\vartheta(\mathbf{x})] = q^{\mathrm{p}}(\mathbf{x}) 
    && \quad \forall \mathbf{x} \in \Gamma^{q} \\
    \label{Eqn:Temp_SS_GE_jump_theta}
    &\llbracket\vartheta(\mathbf{x})\rrbracket = \mathbf{0} 
    && \quad \forall \mathbf{x} \in \Sigma \\ 
    \label{Eqn:Temp_SS_GE_jump_q}
    -&d \, \llbracket
    \mathbf{K}\big(\mathbf{x},\vartheta(\mathbf{x})\big) \, \mathrm{grad}[\vartheta(\mathbf{x})]
    \rrbracket = \chi \,  \mathrm{grad}[\vartheta] \bullet \widehat{\mathbf{t}}(\mathbf{x}) 
    && \quad \forall \mathbf{x} \in \Sigma \\ 
    \label{Eqn:Temp_SS_GE_inlet}
    &\vartheta(\mathbf{x}) = \vartheta_{\mathrm{inlet}}
    && \quad \mbox{at inlet}
\end{alignat}
\end{subequations}

To establish the mentioned principles, we appeal to the Galerkin weak formulation. Accordingly, we use $L_2(\mathcal{K})$ to denote the set of square-integrable functions defined on $\mathcal{K}$. That is, 
\begin{align}
    L_2(\mathcal{K}) := \Big\{u(\mathbf{x}): \mathcal{K} \rightarrow \mathbb{R} \, \Big\vert \, \int_{\mathcal{K}} u^{2}(\mathbf{x}) \; \mathrm{d} \mathcal{K} < +\infty \Big\}
\end{align}
$H^{1}(\mathcal{K})$ denotes the space of square-integrable functions defined on $\mathcal{K}$ with their first (weak) derivatives also square-integrable. Mathematically, 
\begin{align}
    H^1(\mathcal{K}) := \Big\{u(\mathbf{x}) \in L_2(\mathcal{K}) \, \Big\vert \, \int_{\mathcal{K}} \big\|\mathrm{grad}[u(\mathbf{x})]\big\|^{2} \; \mathrm{d} \mathcal{K} < +\infty \Big\}
\end{align}
In the mathematics literature, $L_2(\mathcal{K})$ and $H^{1}(\mathcal{K})$ are called Sobolev spaces, which are themselves Hilbert spaces, and they offer an appropriate functional analytical setting to study partial differential equations \citep{brezis2011functional}.

We then define the following spaces for the trial and test functions, respectively: 
\begin{align}
    \label{Eqn:TProp_trial_function_space}
    \mathcal{U} &:= \left\{\vartheta(\mathbf{x}) \in H^{1}(\Omega) \; \vert \; \vartheta(\mathbf{x}) = \vartheta^\mathrm{p}(\mathbf{x}) \; \mathrm{on} \; \Gamma^\vartheta \; \mathrm{and} \; \vartheta(\mathbf{x})=\vartheta_\mathrm{inlet} \; \mathrm{at} \; s=0 \;\mathrm{on}\; \Sigma\right\} \\
    \label{Eqn:TProp_test_function_space}
    \mathcal{W} &:= \left\{w(\mathbf{x}) \in H^{1}(\Omega) \; \vert \; w(\mathbf{x}) = 0 \; \mathrm{on}\; \Gamma^\vartheta \; \mathrm{and} \; w(\mathbf{x}) = 0 \; \mathrm{at} \; s=0 \; \mathrm{on} \; \Sigma \right\}
\end{align}
Under steady-state conditions, the Galerkin weak formulation reads: Find $\vartheta(\mathbf{x}) \in \mathcal{U}$ such that we have 
\begin{align}
    \label{Eqn:TProp_Galerkin_formulation}
    &\int_{\Omega} d \, \mathrm{grad}[w(\mathbf{x})] \bullet \mathbf{K}(\mathbf{x},\vartheta(\mathbf{x})) \, \mathrm{grad}[\vartheta(\mathbf{x})] \,  \mathrm{d} \Omega  
    +\int_{\Omega} h_T\,w(\mathbf{x})\,(\vartheta(\mathbf{x})-\vartheta_\mathrm{amb})\,  \mathrm{d}\Omega 
    \nonumber \\
    &\qquad +\int_{\Omega} \varepsilon \, \sigma \, w(\mathbf{x})\,\big(\vartheta^{4}(\mathbf{x})-\vartheta^{4}_\mathrm{amb}\big)\,  \mathrm{d}\Omega  
    +\int_{\Sigma} \chi\, w(\mathbf{x})\;\mathrm{grad}[\vartheta(\mathbf{x})]\bullet \widehat{\mathbf{t}}(\mathbf{x})\, \mathrm{d}\Gamma 
    \nonumber \\
    &\qquad \qquad = \int_{\Omega} w(\mathbf{x})\;f(\mathbf{x})\,\mathrm{d}\Omega - 
    \int_{\Gamma^q} w(\mathbf{x})\; q^\mathrm{p}(\mathbf{x})\, \mathrm{d}\Gamma\,\,\,\,\forall w(\mathbf{x})\in \mathcal{W}
\end{align}
where $w(\mathbf{x})$ represents the weighting (or test) function.

\begin{theorem}[minimum principle]
\label{Thm:TProp_Minimum_principle}
Let $\vartheta(\mathbf{x})$ be a non-negative solution of the boundary value problem \eqref{Eqn:Temp_SS_GE_BoE}--\eqref{Eqn:Temp_SS_GE_inlet} under $f(\mathbf{x})\geq 0$, $q^\mathrm{p}(\mathbf{x})\leq 0$, and $\vartheta_{\mathrm{amb}} > 0$. Then, the solution field is bounded below by 
\begin{align}
    \mathrm{min}\Big[\vartheta_\mathrm{amb}, \vartheta_\mathrm{inlet}, \min_{\mathbf{x}\in\Gamma^{\vartheta}}[\vartheta^{\mathrm{p}}(\mathbf{x})]\Big] \leq 
    \vartheta(\mathbf{x})
    \;\; \forall \mathbf{x}\in \overline {\Omega}
\end{align}
\end{theorem}
\begin{proof}
For convenience, we denote the targeted lower bound by 
\begin{align}
    \label{Eqn:TProp_Phi_min}
    \Phi_{\mathrm{min}} := \min\Big[\vartheta_\mathrm{amb}, \vartheta_\mathrm{inlet}, \min_{\mathbf{x}\in\Gamma^{\vartheta}}[\vartheta^{\mathrm{p}}(\mathbf{x})]\Big]
\end{align}
We also define the following field variable: 
\begin{align}
    \label{Eqn:_TProp_beta_definition}
    \beta(\mathbf{x}) := \min\Big[0, \vartheta(\mathbf{x})-\Phi_{\mathrm{min}}\Big]
\end{align}
The above definitions imply that establishing the minimum principle is equivalent to showing that 
\begin{align}
\label{Eqn:_TProp_beta_zero}
\beta(\mathbf{x}) = 0 \quad \forall \mathbf{x} \in \overline{\Omega}
\end{align}

To prove that $\beta(\mathbf{x})$ vanishes, we first note the following properties, which directly stem from Eqs.~\eqref{Eqn:_TProp_beta_definition} and \eqref{Eqn:_TProp_beta_zero}: 
\begin{align}
    \label{Eqn:TProp_beta_property_1}
    &\beta(\mathbf{x}) \leq 0 \quad \forall \mathbf{x}\in \overline {\Omega} \\
    \label{Eqn:TProp_beta_property_2}
    &\beta(\mathbf{x}) = 0 \quad \forall \mathbf{x}\in \Gamma^{\vartheta} \\
    \label{Eqn:TProp_beta_property_3}
    &\beta(\mathbf{x}) = 0 \quad \mbox{at the inlet (i.e., $s=0$ on $\Sigma$)}
\end{align}
Further, at any spatial point, we have 
\begin{align}
    \label{Eqn:TProp_beta_property_4}
    \mathrm{either} \; \beta(\mathbf{x}) = 0 \; \mathrm{or} \; \vartheta(\mathbf{x}) = \beta(\mathbf{x})+\Phi_{\min}
\end{align}
The second property (i.e., Eq.~\eqref{Eqn:TProp_beta_property_2}) implies that $\beta(\mathbf{x}) \in \mathcal{W}$. Taking $w(\mathbf{x}) = \beta(\mathbf{x})$ in Eq.~\eqref{Eqn:TProp_Galerkin_formulation}, we arrive at the following: 
\begin{align}
    \label{Eqn:TProp_MP_step1}
    &\int_{\Omega} d \, \mathrm{grad}[\beta(\mathbf{x})] \bullet \mathbf{K}(\mathbf{x},\vartheta(\mathbf{x})) \, \mathrm{grad}[\vartheta(\mathbf{x})] \,  \mathrm{d} \Omega 
    +\int_{\Omega} h_T\,\beta(\mathbf{x})\,(\vartheta(\mathbf{x})-\vartheta_\mathrm{amb})\,  \mathrm{d}\Omega 
    \nonumber \\ 
    &\hspace{1in} + \int_{\Omega} \varepsilon \, \sigma 
    \, \beta(\mathbf{x}) \, \big(\vartheta^{4}(\mathbf{x}) - \vartheta_{\mathrm{amb}}^4 \big) \, \mathrm{d} \Omega 
    +\int_{\Sigma} \chi\, \beta(\mathbf{x})\;\mathrm{grad}[\vartheta(\mathbf{x})]\bullet \widehat{\mathbf{t}}(\mathbf{x})\, \mathrm{d}\Gamma \nonumber \\
    &\hspace{1in} = \int_{\Omega} \beta(\mathbf{x})\;f(\mathbf{x})\,\mathrm{d}\Omega - 
    \int_{\Gamma^q} \beta(\mathbf{x})\; q^\mathrm{p}(\mathbf{x})\, \mathrm{d}\Gamma
\end{align}
Noting that $f(\mathbf{x}) \geq 0$, $q^{\mathrm{p}}(\mathbf{x}) \leq 0$, and the non-positivity of $\beta(\mathbf{x})$ (i.e., the first property: Eq.~\eqref{Eqn:TProp_beta_property_1}), we conclude that the two integrals on the right side of Eq.~\eqref{Eqn:TProp_MP_step1} are non-positive. We thus write the following inequality: 
\begin{align}
    \label{Eqn:TProp_MP_step2}
    &\int_{\Omega} d \, \mathrm{grad}[\beta(\mathbf{x})] \bullet \mathbf{K}(\mathbf{x},\vartheta(\mathbf{x})) \, \mathrm{grad}[\vartheta(\mathbf{x})] \,  \mathrm{d} \Omega 
    +\int_{\Omega} h_T\,\beta(\mathbf{x})\,(\vartheta(\mathbf{x})-\vartheta_\mathrm{amb})\,  \mathrm{d}\Omega 
    \nonumber \\ 
    &\hspace{1in} + \int_{\Omega} \varepsilon \, \sigma 
    \, \beta(\mathbf{x}) \, \big(\vartheta^{4}(\mathbf{x}) - \vartheta_{\mathrm{amb}}^4 \big) \, \mathrm{d} \Omega 
    +\int_{\Sigma} \chi\, \beta(\mathbf{x})\;\mathrm{grad}[\vartheta(\mathbf{x})]\bullet \widehat{\mathbf{t}}(\mathbf{x})\, \mathrm{d}\Gamma 
    \leq 0
\end{align}
In the above equation, we note that any term that contains $\vartheta(\mathbf{x})$ also carries $\beta(\mathbf{x})$. The fourth property (i.e., Eq.~\eqref{Eqn:TProp_beta_property_4}) implies that, at any spatial point, if $\vartheta(\mathbf{x})$ is not equal to $\beta(\mathbf{x}) + \Phi_{\min}$ then $\beta(\mathbf{x}) = 0$. Thus, replacing $\vartheta(\mathbf{x})$ by $\beta(\mathbf{x}) + \Phi_{\min}$ will not alter the values of such terms. Consequently, upon undertaking the remarked replacement, Eq.~\eqref{Eqn:TProp_MP_step2} turns into the following:
\begin{align}
    \label{Eqn:TProp_MP_step3}
    &\int_{\Omega} d \, \mathrm{grad}[\beta(\mathbf{x})] \bullet \mathbf{K}(\mathbf{x},\vartheta(\mathbf{x})) \, \mathrm{grad}[\beta(\mathbf{x})] \,  \mathrm{d} \Omega + 
    \int_{\Omega} h_T\,\beta(\mathbf{x})\,(\beta(\mathbf{x}) + \Phi_{\mathrm{min}}-\vartheta_\mathrm{amb})\,  \mathrm{d}\Omega \nonumber \\
    &\qquad + \int_{\Omega} \varepsilon \, \sigma 
    \, \beta(\mathbf{x}) \, \big(\beta(\mathbf{x}) + \Phi_{\min} 
    - \vartheta_{\mathrm{amb}} \big)
    \big(\vartheta(\mathbf{x}) + \vartheta_{\mathrm{amb}} \big)
    \big(\vartheta^{2}(\mathbf{x}) + \vartheta_{\mathrm{amb}}^2 \big) \, \mathrm{d} \Omega 
    \nonumber \\
    &\qquad \qquad +\int_{\Sigma} \chi\, \beta(\mathbf{x})\;\mathrm{grad}[\beta(\mathbf{x})]\bullet \widehat{\mathbf{t}}(\mathbf{x})\, \mathrm{d}\Gamma 
    \leq 0  
\end{align}
Definition \eqref{Eqn:TProp_Phi_min} implies  $(\Phi_{\mathrm{min}}-\vartheta_{\mathrm{amb}}) \leq 0$, which further implies the product $\beta(\mathbf{x}) \; (\Phi_{\mathrm{min}}-\vartheta_{\mathrm{amb}})$ is non-negative. Using this result on the second integral in  Eq.~\eqref{Eqn:TProp_MP_step3}, we establish the following inequality:
\begin{align}
    \label{Eqn:TProp_MP_step4}
    &\int_{\Omega} d \, \mathrm{grad}[\beta(\mathbf{x})] \bullet \mathbf{K}(\mathbf{x},\vartheta(\mathbf{x})) \, \mathrm{grad}[\beta(\mathbf{x})] \,  \mathrm{d} \Omega + 
    \int_{\Omega} h_T\,\beta^2(\mathbf{x}) \,  \mathrm{d}\Omega \nonumber \\
    &\qquad + \int_{\Omega} \varepsilon \, \sigma 
    \, \beta(\mathbf{x}) \, \big(\beta(\mathbf{x}) + \Phi_{\min} 
    - \vartheta_{\mathrm{amb}} \big)
    \big(\vartheta(\mathbf{x}) + \vartheta_{\mathrm{amb}} \big)
    \big(\vartheta^{2}(\mathbf{x}) + \vartheta_{\mathrm{amb}}^2 \big) \, \mathrm{d} \Omega 
    \nonumber \\
    &\qquad \qquad +\int_{\Sigma} \chi\, \beta(\mathbf{x})\;\mathrm{grad}[\beta(\mathbf{x})]\bullet \widehat{\mathbf{t}}(\mathbf{x})\, \mathrm{d}\Gamma 
    \leq 0  
\end{align}
The third integral is non-negative because $\beta(\mathbf{x}) \leq 0$, $\beta(\mathbf{x}) + \Phi_{\min} - \vartheta_{\mathrm{amb}} \leq 0$, $\vartheta(\mathbf{x}) + \vartheta_{\mathrm{amb}} \geq 0$ (since the solution field is non-negative and $\vartheta_{\mathrm{amb}} > 0$), and $\vartheta^{2}(\mathbf{x}) + \vartheta_{\mathrm{amb}}^2 \geq 0$. We thus have
\begin{align}
    \label{Eqn:TProp_MP_step5}
    \int_{\Omega} d \, \mathrm{grad}[\beta(\mathbf{x})] \bullet \mathbf{K}(\mathbf{x},\vartheta(\mathbf{x})) \, \mathrm{grad}[\beta(\mathbf{x})] \,  \mathrm{d} \Omega 
    &+\int_{\Omega} h_T\,\beta^{2}(\mathbf{x}) \,  \mathrm{d}\Omega
    \nonumber \\ 
   &\qquad +\int_{\Sigma} \chi\, \beta(\mathbf{x})\;\mathrm{grad}[\beta(\mathbf{x})]\bullet \widehat{\mathbf{t}}(\mathbf{x})\, \mathrm{d}\Gamma 
    \leq 0  
\end{align}
Since the thermal conductivity is uniformly elliptic (i.e., Eq.~\eqref{Eqn:TProp_thermal_conductivity_positive_definite}), the above inequality implies the following: 
\begin{align}
    \label{Eqn:TProp_MP_step6}
    \int_{\Omega} d \, \mathrm{grad}[\beta(\mathbf{x})] \bullet k_1 \, \mathrm{grad}[\beta(\mathbf{x})] \,  \mathrm{d} \Omega + 
    &\int_{\Omega} h_T\,\beta^{2}(\mathbf{x}) 
    \, \mathrm{d}\Omega
    +\int_{\Sigma} \chi\, \beta(\mathbf{x})\;\mathrm{grad}[\beta(\mathbf{x})]\bullet \widehat{\mathbf{t}}(\mathbf{x})\, \mathrm{d}\Gamma 
    \leq 0 
\end{align}
Integrating the third expression in Eq.~\eqref{Eqn:TProp_MP_step6} along the vasculature and noting $\beta(\mathbf{x})$ vanishes at the inlet (i.e., the third property given by Eq.~\eqref{Eqn:TProp_beta_property_3}), we get the following:
\begin{align}
    \label{Eqn:TProp_MP_step7}
    \int_{\Omega} d \, \mathrm{grad}[\beta(\mathbf{x})] \bullet k_1 \, \mathrm{grad}[\beta(\mathbf{x})] \,  \mathrm{d} \Omega + 
    &\int_{\Omega} h_T\,\beta^{2}(\mathbf{x}) 
    \, \mathrm{d}\Omega +
    \frac{\chi}{2} \, \beta^{2}(\mathbf{x})\Big\vert_{\mathrm{outlet}}
    \leq 0 
\end{align}
All the terms on the left side of the above inequality are non-negative. Hence, each term must be zero, implying that 
\begin{align}
    \label{Eqn:TProp_MP_step8}
    \beta(\mathbf{x}) = 0 \quad \forall \mathbf{x} \in \overline{\Omega}
\end{align}
rendering the minimum principle to be true.
\end{proof}

In the case of linear problems, one can obtain the maximum principle from the minimum principle by rewriting the governing equations in terms of a new variable $\varphi(\mathbf{x}) = - \vartheta(\mathbf{x})$ and running through a similar course of the proof outlined under the minimum principle. For example, see \citep[Theorems 3.1 and 3.3]{nakshatrala2023ROM} for the case of constant material properties. Since the problem at hand is nonlinear (because of the radiation term as well as temperature-dependent material properties), the remarked approach does not work. So, we prove the maximum principle by alternative means, as given below.

\begin{theorem}[maximum principle]
\label{Thm:TProp_Maximum_principle}
Let $\vartheta(\mathbf{x})$ be a non-negative solution of the boundary value problem \eqref{Eqn:Temp_SS_GE_BoE}--\eqref{Eqn:Temp_SS_GE_inlet} under $f(\mathbf{x}) \leq 0$ and $0 \leq q^\mathrm{p}(\mathbf{x})$, and $\vartheta_{\mathrm{amb}} > 0$. Then, the solution field is bounded above by:
\begin{align}
    \vartheta(\mathbf{x}) \leq \mathrm{max}\Big[\vartheta_\mathrm{amb}, \vartheta_\mathrm{inlet}, \max_{\mathbf{x}\in\Gamma^{\vartheta}}[\vartheta^{\mathrm{p}}(\mathbf{x})]\Big]
    \;\; \forall \mathbf{x}\in \overline {\Omega}
\end{align}
\end{theorem}
\begin{proof}
    We denote the upper bound by 
\begin{align}
    \label{Eqn:TProp_Phi_max}
    \Phi_{\mathrm{max}} := \max\Big[\vartheta_\mathrm{amb}, \vartheta_\mathrm{inlet}, \max_{\mathbf{x}\in\Gamma^{\vartheta}}[\vartheta^{\mathrm{p}}(\mathbf{x})]\Big]
\end{align}
We also define the following field variable:
\begin{align}
    \label{Eqn:TProp_beta_max_definition}
    \beta(\mathbf{x}) := \max\Big[0,
    \vartheta (\mathbf{x}) 
    - \Phi_{\mathrm{max}}\Big]
\end{align}
To prove the maximum principle, it suffices to show that 
\begin{align}
\label{Eqn:_TProp_beta_max_zero}
\beta(\mathbf{x}) = 0 \quad \forall \mathbf{x} \in \overline{\Omega}
\end{align}

We proceed by noting the following properties that $\beta(\mathbf{x})$ satisfies: 
\begin{align}
    \label{Eqn:TProp_beta_max_property_1}
    &\beta(\mathbf{x}) \geq 0 \quad \forall \mathbf{x}\in \overline {\Omega} \\
    \label{Eqn:TProp_beta_max_property_2}
    &\beta(\mathbf{x}) = 0 \quad \forall \mathbf{x}\in \Gamma^{\vartheta} \\
    \label{Eqn:TProp_beta_max_property_3}
    &\beta(\mathbf{x}) = 0 \quad \mbox{at the inlet (i.e., $s=0$ on $\Sigma$)}
\end{align}
Also, at any spatial point, we have 
\begin{align}
    \label{Eqn:TProp_beta_max_property_4}
    \mathrm{either} \; \beta(\mathbf{x}) = 0 \; \; \mathrm{or} \; \;      \vartheta(\mathbf{x}) = \beta(\mathbf{x})+\Phi_{\max}
\end{align}

Many steps below will be similar to those under the minimum principle, but accounting for the differences (cf. Eqs.~\eqref{Eqn:TProp_Phi_min} and \eqref{Eqn:TProp_Phi_max}, \eqref{Eqn:_TProp_beta_definition} and \eqref{Eqn:TProp_beta_max_definition}, and \eqref{Eqn:TProp_beta_property_1} and \eqref{Eqn:TProp_beta_max_property_1}). The second property implies that $\beta(\mathbf{x}) \in \mathcal{W}$. Taking $w(\mathbf{x}) = \beta(\mathbf{x})$ in Eq.~\eqref{Eqn:TProp_Galerkin_formulation}, we write 
\begin{align}
    \label{Eqn:TProp_MaxP_step1}
    &\int_{\Omega} d \, \mathrm{grad}[\beta(\mathbf{x})] \bullet \mathbf{K}(\mathbf{x},\vartheta(\mathbf{x})) \, \mathrm{grad}[\vartheta(\mathbf{x})] \,  \mathrm{d} \Omega 
    +\int_{\Omega} h_T\,\beta(\mathbf{x})\,(\vartheta(\mathbf{x})-\vartheta_\mathrm{amb})\,  \mathrm{d}\Omega 
    \nonumber \\ 
    &\hspace{1in} + \int_{\Omega} \varepsilon \, \sigma 
    \, \beta(\mathbf{x}) \, \big(\vartheta^{4}(\mathbf{x}) - \vartheta_{\mathrm{amb}}^4 \big) \, \mathrm{d} \Omega 
    +\int_{\Sigma} \chi\, \beta(\mathbf{x})\;\mathrm{grad}[\vartheta(\mathbf{x})]\bullet \widehat{\mathbf{t}}(\mathbf{x})\, \mathrm{d}\Gamma \nonumber \\
    &\hspace{1in} = \int_{\Omega} \beta(\mathbf{x})\;f(\mathbf{x})\,\mathrm{d}\Omega - 
    \int_{\Gamma^q} \beta(\mathbf{x})\; q^\mathrm{p}(\mathbf{x})\, \mathrm{d}\Gamma
\end{align}
Since $f(\mathbf{x}) \leq 0$, $q^{\mathrm{p}}(\mathbf{x}) \geq 0$, the non-negativity of $\beta(\mathbf{x})$ (i.e., Eq.~\eqref{Eqn:TProp_beta_max_property_1}) implies the two integrals on the right side of Eq.~\eqref{Eqn:TProp_MaxP_step1} are non-positive. We thus have the following inequality:
\begin{align}
    \label{Eqn:TProp_MaxP_step2}
    &\int_{\Omega} d \, \mathrm{grad}[\beta(\mathbf{x})] \bullet \mathbf{K}(\mathbf{x},\vartheta(\mathbf{x})) \, \mathrm{grad}[\vartheta(\mathbf{x})] \,  \mathrm{d} \Omega 
    +\int_{\Omega} h_T\,\beta(\mathbf{x})\,(\vartheta(\mathbf{x})-\vartheta_\mathrm{amb})\,  \mathrm{d}\Omega 
    \nonumber \\ 
    &\hspace{1in} + \int_{\Omega} \varepsilon \, \sigma 
    \, \beta(\mathbf{x}) \, \big(\vartheta^{4}(\mathbf{x}) - \vartheta_{\mathrm{amb}}^4 \big) \, \mathrm{d} \Omega 
    +\int_{\Sigma} \chi\, \beta(\mathbf{x})\;\mathrm{grad}[\vartheta(\mathbf{x})]\bullet \widehat{\mathbf{t}}(\mathbf{x})\, \mathrm{d}\Gamma 
   \leq 0
\end{align}
In view of the fourth property ~\eqref{Eqn:TProp_beta_max_property_4}, we replace $\vartheta(\mathbf{x})$ with $\beta(\mathbf{x}) + \Phi_{\max}$ in Eq.~\eqref{Eqn:TProp_MaxP_step2} and arrive at the following inequality:
\begin{align}
    \label{Eqn:TProp_MaxP_step3}
    &\int_{\Omega} d \, \mathrm{grad}[\beta(\mathbf{x})] \bullet \mathbf{K}(\mathbf{x},\vartheta(\mathbf{x})) \, \mathrm{grad}[\beta(\mathbf{x})] \,  \mathrm{d} \Omega + 
    \int_{\Omega} h_T\,\beta(\mathbf{x})\,(\beta(\mathbf{x}) + \Phi_{\mathrm{max}}-\vartheta_\mathrm{amb})\,  \mathrm{d}\Omega \nonumber \\
    &\qquad + \int_{\Omega} \varepsilon \, \sigma 
    \, \beta(\mathbf{x}) \, \big(\beta(\mathbf{x}) + \Phi_{\max} 
    - \vartheta_{\mathrm{amb}} \big)
    \big(\vartheta(\mathbf{x}) + \vartheta_{\mathrm{amb}} \big)
    \big(\vartheta^{2}(\mathbf{x}) + \vartheta_{\mathrm{amb}}^2 \big) \, \mathrm{d} \Omega 
    \nonumber \\
    &\qquad \qquad +\int_{\Sigma} \chi\, \beta(\mathbf{x})\;\mathrm{grad}[\beta(\mathbf{x})]\bullet \widehat{\mathbf{t}}(\mathbf{x})\, \mathrm{d}\Gamma 
    \leq 0  
\end{align}
Definition \eqref{Eqn:TProp_Phi_max} implies  $(\Phi_{\mathrm{max}}-\vartheta_{\mathrm{amb}}) \geq 0$, which further implies the product $\beta(\mathbf{x}) \; (\Phi_{\mathrm{max}}-\vartheta_{\mathrm{amb}})$ is positive. Using this result on the second and third integrals in  Eq.~\eqref{Eqn:TProp_MaxP_step3}, we realize the following inequality:
\begin{align}
    \label{Eqn:TProp_MaxP_step4}
    &\int_{\Omega} d \, \mathrm{grad}[\beta(\mathbf{x})] \bullet \mathbf{K}(\mathbf{x},\vartheta(\mathbf{x})) \, \mathrm{grad}[\beta(\mathbf{x})] \,  \mathrm{d} \Omega + 
    \int_{\Omega} h_T\,\beta^2(\mathbf{x}) \,  \mathrm{d}\Omega \nonumber \\
    &\qquad + \int_{\Omega} \varepsilon \, \sigma 
    \, \beta^2(\mathbf{x}) \, 
    \big(\vartheta(\mathbf{x}) + \vartheta_{\mathrm{amb}} \big)
    \big(\vartheta^{2}(\mathbf{x}) + \vartheta_{\mathrm{amb}}^2 \big) \, \mathrm{d} \Omega 
     +\int_{\Sigma} \chi\, \beta(\mathbf{x})\;\mathrm{grad}[\beta(\mathbf{x})]\bullet \widehat{\mathbf{t}}(\mathbf{x})\, \mathrm{d}\Gamma 
    \leq 0  
\end{align}
Since $\vartheta(\mathbf{x})$ and $\vartheta_{\mathrm{amb}}$ are non-negative, the third integral is non-negative. We thus have
\begin{align}
    \label{Eqn:TProp_MaxP_step5}
    \int_{\Omega} d \, \mathrm{grad}[\beta(\mathbf{x})] \bullet \mathbf{K}(\mathbf{x},\vartheta(\mathbf{x})) \, \mathrm{grad}[\beta(\mathbf{x})] \,  \mathrm{d} \Omega  
    &+\int_{\Omega} h_T\,\beta(\mathbf{x})^2 \,  \mathrm{d}\Omega
    \nonumber \\
   &\quad +\int_{\Sigma} \chi\, \beta(\mathbf{x})\;\mathrm{grad}[\beta(\mathbf{x})]\bullet \widehat{\mathbf{t}}(\mathbf{x})\, \mathrm{d}\Gamma 
    \leq 0  
\end{align}
Invoking the uniform ellipticity of the conductivity tensor (i.e., Eq.~\eqref{Eqn:TProp_thermal_conductivity_positive_definite}), integrating the third expression in Eq.~\eqref{Eqn:TProp_MaxP_step5}, and simplifying the resulting expression using the third property ~\eqref{Eqn:TProp_beta_max_property_3}, we get:
\begin{align}
    \label{Eqn:TProp_MaxP_step6}
    \int_{\Omega} d \, \mathrm{grad}[\beta(\mathbf{x})] \bullet k_1 \, \mathrm{grad}[\beta(\mathbf{x})] \,  \mathrm{d} \Omega + 
    &\int_{\Omega} h_T\,\beta^{2}(\mathbf{x}) 
    \, \mathrm{d}\Omega +
    \frac{\chi}{2} \, \beta^{2}(\mathbf{x})\Big\vert_{\mathrm{outlet}}
    \leq 0 
\end{align}
Similar to the case in the minimum principle, all the terms on the left side of the above inequality are non-negative. Hence, each term must be zero, implying that 
\begin{align}
    \label{Eqn:TProp_MaxP_step7}
    \beta(\mathbf{x}) = 0 \quad \forall \mathbf{x} \in \overline{\Omega}
\end{align}
Thus, the maximum principle is established.
\end{proof}

The two theorems presented above show that the response to the initial question (Q1) is in the affirmative---the minimum and maximum principles hold even with the chosen temperature-dependent material properties under a steady-state response.

\begin{remark}
For the proof of minimum and maximum principles, if the radiation is absent (i.e., $\varepsilon = 0$), then the ``non-negative'' requirement for the solution field can be dropped from the hypothesis of Theorems \ref{Thm:TProp_Minimum_principle} and \ref{Thm:TProp_Maximum_principle}. 
\end{remark}

\section{REPRESENTATIVE TEST PROBLEMS}
\label{Sec:S5_TProp_NR}
We now outline canonical problems, guided by previous active-cooling experiments (e.g., \citep{devi2021microvascular,Nakshatrala_PNAS_Nexus_2023}), that will be used to show how the temperature-dependent material properties affect the solution fields, quantitatively and qualitatively.

Consider a square computational domain of size $100 \, \mathrm{mm} \times 100 \, \mathrm{mm}$ with thickness $d = 5$ mm. The entire lateral boundary is adiabatic (i.e., $q^{\mathrm{p}}(\mathbf{x},t) = 0$ and $\Gamma^{q} = \partial \Omega$), a source supplies heat to the bottom surface, and the top surface is free to convect and radiate. The domain consists of an embedded vasculature. A fluid flows through the vasculature, starting from the inlet and exiting the domain at the outlet. To discern the effect of temperature-dependent material properties on thermal response, we have considered
\begin{enumerate}[(i)]
\item three distinct vasculature layouts (i.e., U-shaped, serpentine, and asymmetric), as shown in \textbf{Fig.~\ref{Fig4:TProp_Problem_description}},
\item three materials---CFRP, GFRP, and epoxy---to encompass a diverse thermal conductivity range, 
\item  two heat fluxes applied separately: $1000$ and $2000 \; \mathrm{W/m^2}$, and 
\item a representative flow rate of $Q = 1 \; \mathrm{mL/min}$.
\end{enumerate}

\begin{figure}[h]
    \centering
    \includegraphics[scale=0.4]{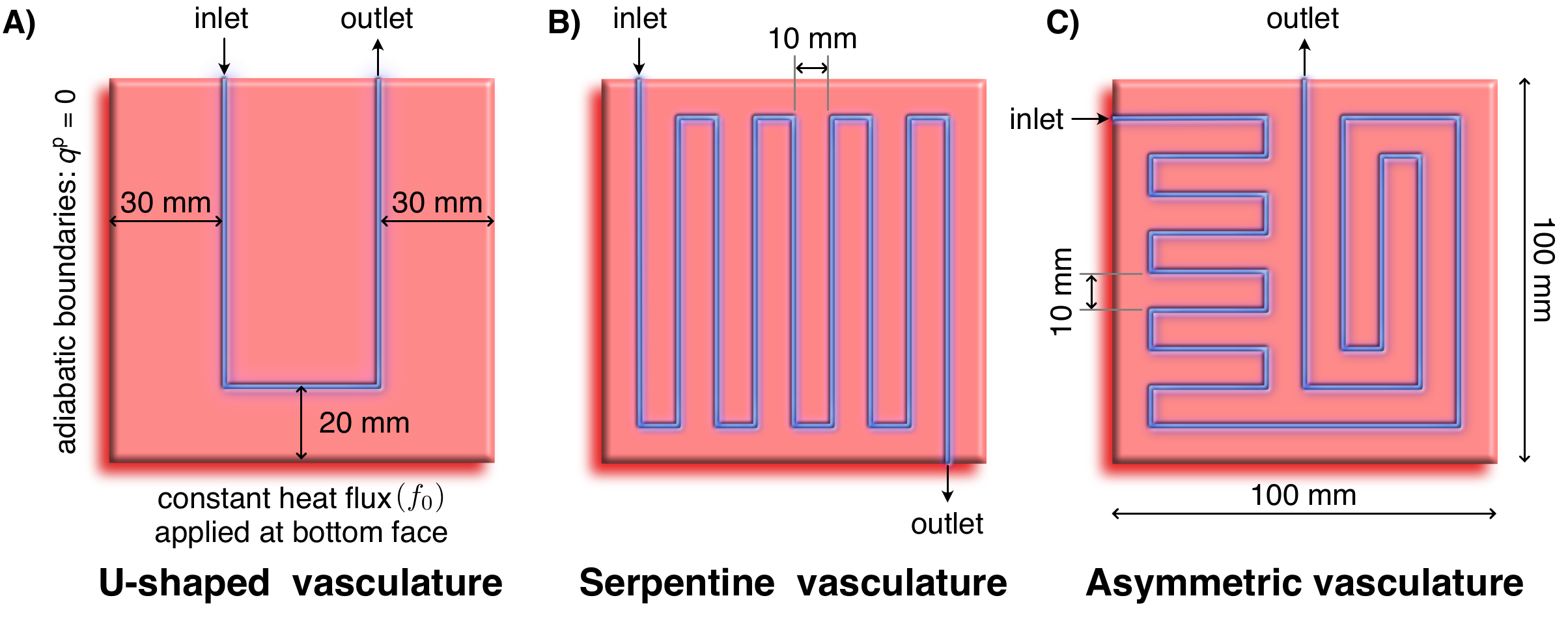}
    \caption{\textsf{Problem description.} The figure shows three different vascular layouts used in the study: \textbf{A)} U-shaped, \textbf{B)} serpentine, and \textbf{C)} asymmetric. Blue lines represent channels, which begin at the inlet and end at the outlet. In all cases, the computational domain is a square of size $100 \times 100 \, \mathrm{mm}$ with thickness $d = 5 \, \mathrm{mm}$. Active cooling is achieved by flowing fluid into the vasculature at the inlet with fluid's temperature equal to the ambient (i.e., $\vartheta_{\mathrm{inlet}} = \vartheta_{\mathrm{amb}}$). At the bottom face, a source supplies a uniform heat flux (i.e., $f(\mathbf{x}) = f_0$). The top surface is free to exchange heat via convection and radiation. The lateral boundaries are thermally isolated (i.e., adiabatic). All the figures are drawn to the scale. 
    \label{Fig4:TProp_Problem_description}}
\end{figure}
In all our numerical simulations, we have used the single-field Galerkin weak formulation, the finite element method for spatial discretization, and the backward difference formula (BDF) with the options of maximum order five and minimum order one. We have implemented the weak formulation using the ``weak form PDE'' utility available in \citet{COMSOL} and used three-node triangular elements with second-order Lagrange shape functions. Table \ref{Table:TProp_param_values} provides the values employed in the numerical simulations.\footnote{Our selection of the value for the total time $\mathcal{T}$, provided in Table \ref{Table:TProp_param_values}, is based on the active-cooling experiments reported in \citep{devi2021microvascular,Nakshatrala_PNAS_Nexus_2023}. By $\mathcal{T} = 1500 \, \mathrm{s}$ in those experiments, the evolution of the temperature field flattened, indicating that the system had been closer to the steady state.}

\begin{table}[ht]
    \centering 
    \caption{Parameters used in numerical simulations.} 
    \begin{tabular}{||c|c||}
        \hline
        \textbf{Parameter} & \textbf{Value} \\ [0.5ex] 
        \hline\hline
        Fluid (coolant) & Water  
        \\ \hline 
        Fluid's density ($\rho_f$) & 1000 $\mathrm{kg/m^3}$ 
        \\ \hline
        Fluid's specific heat capacity ($c_f$) & 4183 $\mathrm{J/(kg \cdot K)}$ 
        \\ \hline 
        Volumetric flow rate ($Q$) & 1 $\mathrm{mL/min}$ 
        \\ \hline \hline 
        Host material & CFRP, GFRP, or epoxy 
        \\ \hline 
        Specific heat capacity ($c_s$) & Figure \ref{Fig2:TProp_Variation_of_host_properties}A
        \\ \hline 
        Thermal conductivity ($k_{s}$) & Figure \ref{Fig2:TProp_Variation_of_host_properties}B
        \\ \hline \hline 
        Heat transfer coefficient ($h_T$) & 21 $\mathrm{W/m^2}$  \citep{devi2021microvascular}
        \\ \hline
        Emissivity coefficient ($\varepsilon$) & 0.97   
        \\ \hline
        Stefan-Boltzmann constant ($\sigma$) & $5.67 \times 10^{-8} \; \mathrm{W/(m^3 \cdot K^4)}$   
        \\ \hline
        Applied heat flux ($f_0$) & 1000 or 2000 
        $\mathrm{W/m^2}$ 
        \\ \hline 
        Ambient temperature ($\vartheta_{\mathrm{amb}}$) & 296.42 $\mathrm{K}$ \\
        \hline \hline 
        Time-stepping scheme & backward difference formula (BDF) \\
        \hline
        Time step ($\Delta t$) & 1 sec \\ 
        \hline
        Total time ($\mathcal{T}$) & 1500 secs \\
        \hline
    \end{tabular}
    \label{Table:TProp_param_values} 
\end{table}

\pagebreak
\section{EFFECT ON THERMAL INVARIANTS UNDER FLOW REVERSAL}
\label{Sec:S6_TProp_Flow_reversal}
Recently, \citet{Nakshatrala_PNAS_Nexus_2023} have shown that the mean surface temperature (MST) and thermal efficiency ($\eta$) remain invariant under flow reversal (i.e., swapping inlet and outlet). However, their analysis assumes material properties are independent of temperature. This section examines how the temperature-dependent material properties affect these two invariants.

For the temperature field $\vartheta(\mathbf{x},t)$, the MST is defined as follows: 
\begin{align}
    \label{Eqn:TProp_MST_definition}
    \vartheta_{\mathrm{MST}}(t) = \frac{1}{\mathrm{meas}(\Omega)} \int_{\Omega} \vartheta(\mathbf{x},t) \, \mathrm{d} \Omega 
\end{align}
in which $\mathrm{meas}(\Omega)$ denotes the set measure of $\Omega$ (i.e., the area of the domain). The thermal efficiency for active cooling applications is defined as the ratio of the rate of heat extracted by the coolant to the total heat supplied by the heat source, for example, see \citep{nakshatrala2023ROM}. Mathematically, 
\begin{align}
    \label{Eqn:TProp_Efficiency_definition}
    \eta := \frac{\mbox{rate of heat extracted by the coolant}}{\mbox{rate of heat supplied by the heater}}
\end{align}
where $\eta$ denotes the thermal efficiency. If the lateral boundaries are adiabatic, the efficiency takes the following form: 
\begin{align}
    \label{Eqn:TProp_Efficiency_Adiabatic_boundaries}
    \eta(t)
    = \left(\int_{\Omega} f(\mathbf{x},t) \, \mathrm{d} \Omega \right)^{-1} \chi \, \left(\vartheta_{\mathrm{outlet}}(t) - \vartheta_{\mathrm{inlet}}\right)
\end{align}
where $f(\mathbf{x},t)$ denotes the applied heat source, $\chi$ is the heat capacity rate, and $\vartheta_{\mathrm{outlet}}(t)$ and $\vartheta_{\mathrm{inlet}}$ depict the outlet and inlet temperatures, respectively. Note that the outlet temperature varies with time, and so does the efficiency. When the applied heat is constant, $f(\mathbf{x},t) = f_0$, the efficiency takes the following simplified form: 
\begin{align}
    \label{Eqn:TProp_Efficiency_special_case}
    \eta(t)
    = \frac{\chi \, \left(\vartheta_{\mathrm{outlet}}(t) - \vartheta_{\mathrm{inlet}}\right)}{\mathrm{meas}(\Omega) \, f_0}
\end{align}
The above expression reveals that for a constant inlet temperature, the efficiency is proportional to the outlet temperature, implying that if the efficiency is invariant, so is the outlet temperature.

Using the above notation, the two invariants under flow reversal take the following mathematical form: 
\begin{itemize}
    \item Invariant \#1 --- the mean surface temperature: 
    \begin{align}
        \vartheta^{(f)}_{\mathrm{MST}}(t) = 
        \vartheta^{(r)}_{\mathrm{MST}}(t) \; \forall t
    \end{align}
 
    \item Invariant \#2 --- the outlet temperature: 
    \begin{align}
        \vartheta^{(f)}_{\mathrm{outlet}}(t) = \vartheta^{(r)}_{\mathrm{outlet}}(t)\; \forall t
    \end{align}
\end{itemize}
In the above equations, the superscripts ``$(f)$" and ``$(r)$" represent the forward and reverse flows, respectively. 

These invariants were mathematically derived assuming constant material properties \citep{Nakshatrala_PNAS_Nexus_2023}. Below, we investigate how these invariants under flow reversal are affected when temperature-dependent material properties are accounted for.

\subsection{Numerical verification}
\textbf{Figure~\ref{Fig5:TProp_Flow_reversal_MST_vs_time}} shows how the mean surface temperatures vary over time for three vascular layouts. Three key observations are: i) The MST remains invariant under flow reversal even when the material properties depend on the temperature. ii) There is no visible disparity in MST between the constant material properties (CMP) versus the temperature-dependent material properties (TDMP) cases for all three material systems under a lower heat flux (i.e., $1000 \; \mathrm{W/m^2}$). iii) Variations in the MST between the CMP and TDMP cases are evident, particularly noticeable for GFRP in the transient regime when subjected to a higher heat flux of $2000 \; \mathrm{W/m^2}$. These distinctions become negligible when the system approaches a steady state. 

\textbf{Figure~\ref{Fig6:TProp_Outlet_temperature_vs_time}} shows the temporal evolution of outlet temperatures. When operating at a lower heat flux ($1000 \; \mathrm{W/m^2}$), the disparities in outlet temperature between the CMP and TDMP cases are insignificant. In contrast, when exposed to a double heat flux level ($2000 \; \mathrm{W/m^2}$), there are comparatively more pronounced differences, particularly during the middle time period, especially for the GFRP. Nevertheless, as the system approaches a steady state, these temperature deviations gradually diminish.

The above results answer the second question (Q2). In the chosen temperature range, the mean surface temperature and the outlet temperature remain invariant under flow reversal even when the material properties depend on the temperature, similar to the previously shown result under constant material properties.

\begin{figure}[h]
    \centering
    \includegraphics[scale=0.26]{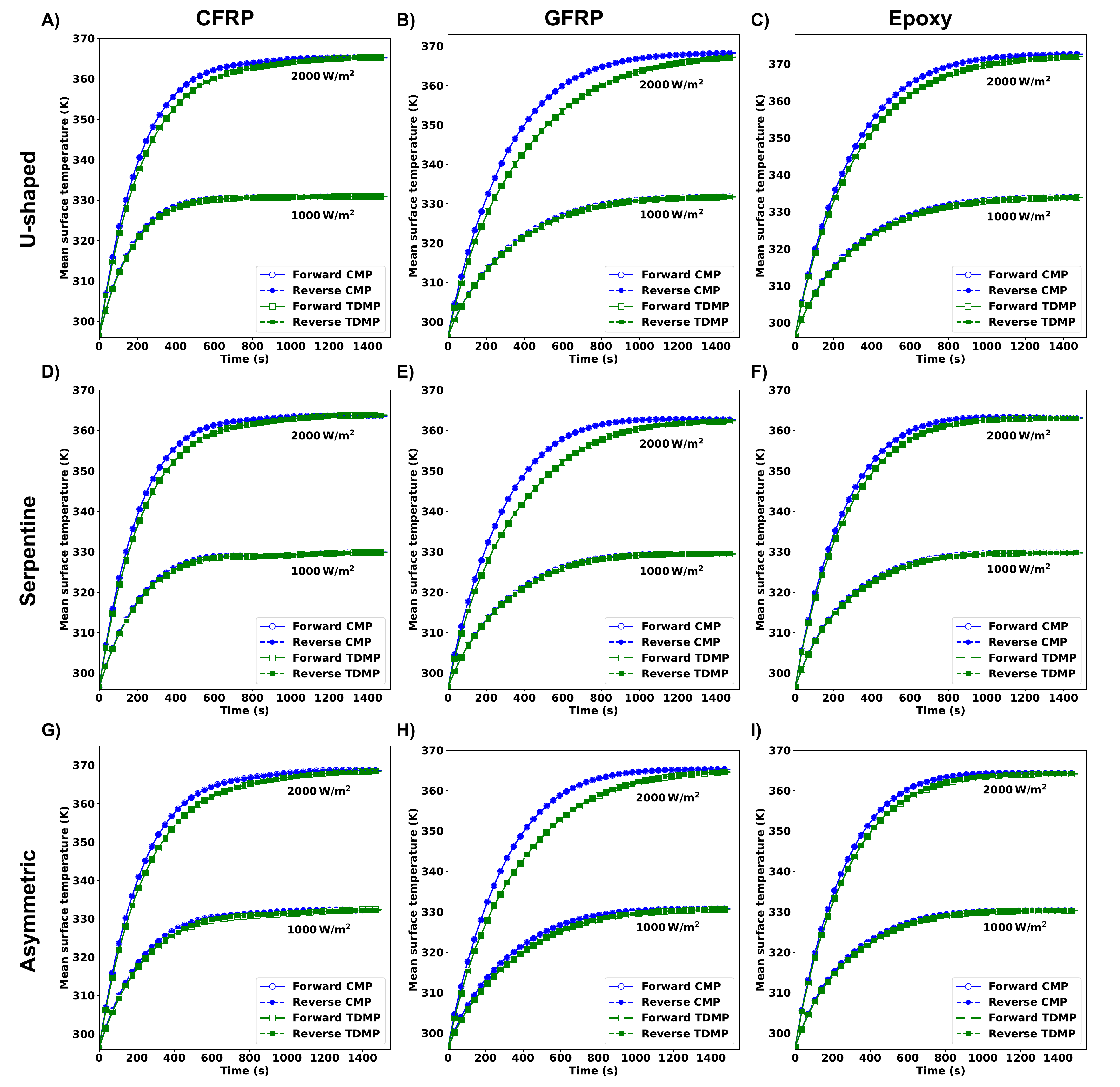}
    \caption{\textsf{Flow reversal --- Mean surface temperature}.
    This figure shows the mean surface temperature over time under forward and reverse flow conditions and for CMP and TDMP cases. We have considered all three vasculature layouts and two applied heat fluxes: 1000 and 2000  $\mathrm{W/m^2}$. From the reported results, we draw two inferences. First, the mean surface temperature (i.e., the first thermal invariant) exhibits negligible change under flow reversal, even when the material properties depend on temperature. Second, minor differences exist between the mean surface temperatures under TDMP and CMP cases, comparatively more in GFRP, in the temperature range considered in this paper. Ultimately, the variation diminishes as the system approaches steady-state. \label{Fig5:TProp_Flow_reversal_MST_vs_time}}
\end{figure}

\begin{figure}[h]
    \centering
    \includegraphics[scale=0.26]{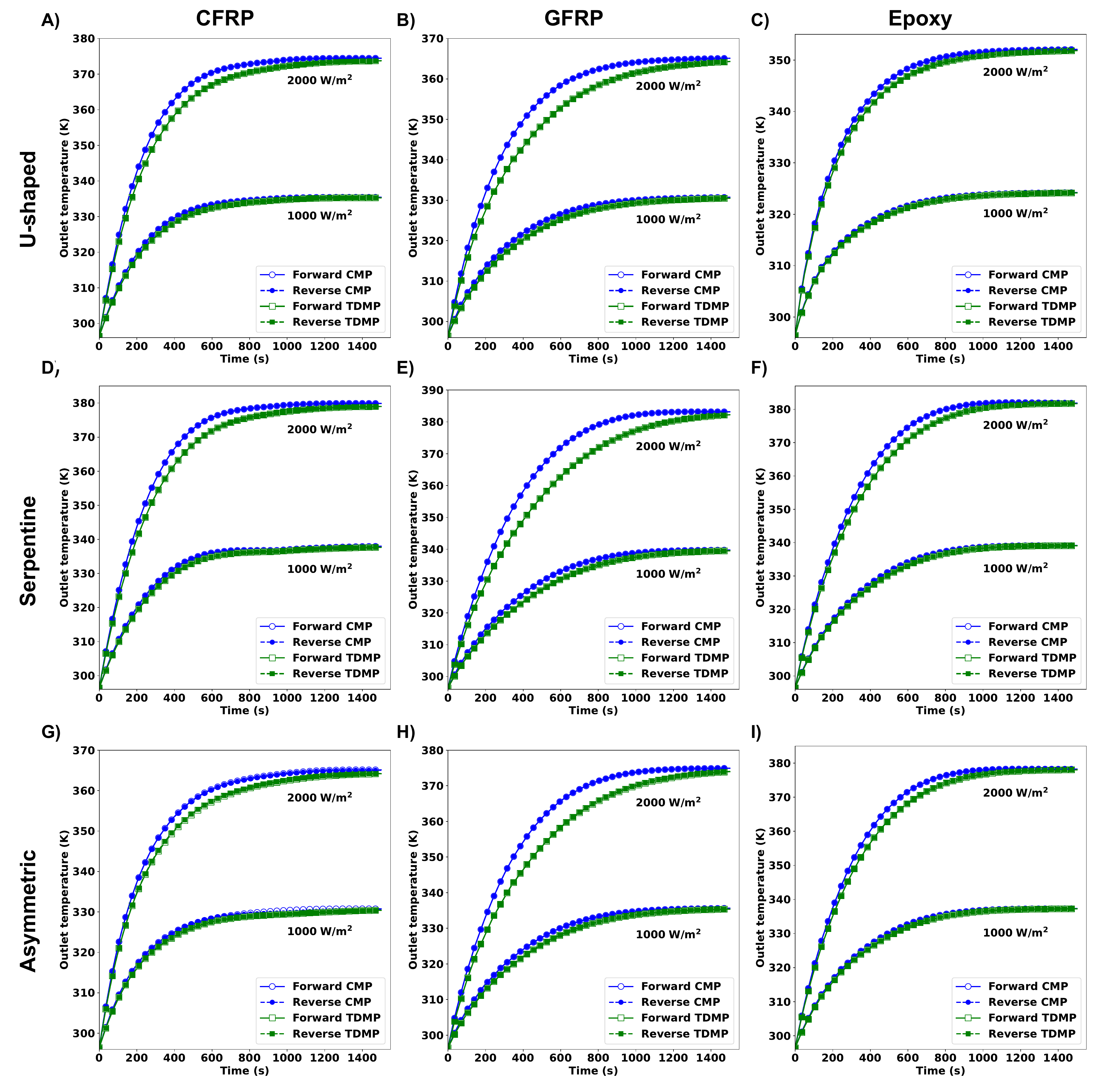}
    \caption{\textsf{Flow reversal --- Outlet temperature}. This figure shows the outlet temperature over time under forward and reverse flow conditions and for CMP and TDMP cases. We have considered all three vasculature layouts and two applied heat fluxes: 1000 and 2000  $\mathrm{W/m^2}$. From the reported results, we draw two inferences. First, the outlet temperature (i.e., the second thermal invariant) exhibits negligible change under flow reversal, even when the material properties depend on temperature. Second, disparities can be observed in the transient state between the outlet temperatures in TDMP and CMP scenarios, particularly in the case of GFRP. However, these differences tend to decrease as the system approaches a steady state.
    \label{Fig6:TProp_Outlet_temperature_vs_time}}
\end{figure}

\section{QUANTITATIVE COMPARISONS}
\label{Sec:S7_TProp_Quantitative}

\textbf{Figure~\ref{Fig7:TProp_Temp_profile}} collates temperature and heat flux vector profiles for the CMP (measured at $23 \; ^\circ \mathrm{C}$) and TDMP cases. For all three vasculature layouts, the corresponding profiles do not exhibit notable differences, either qualitatively or quantitatively. 

\begin{figure}[h]
    \centering
    \includegraphics[scale=0.95]{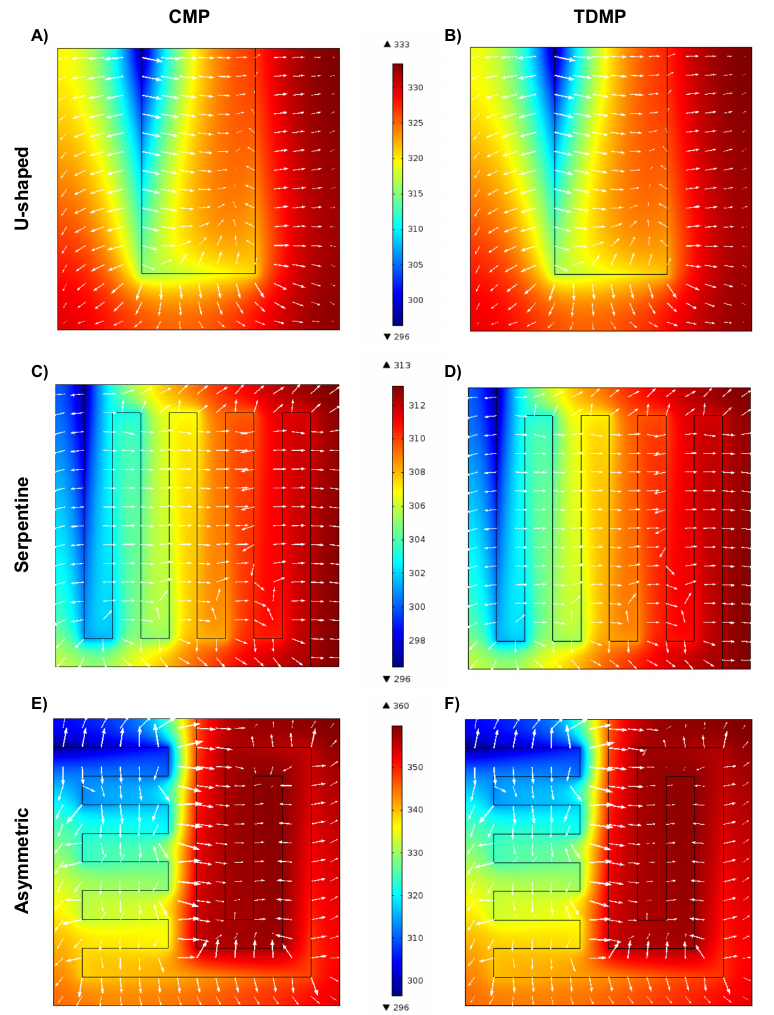}
    \caption{\textsf{Temperature contours and heat flux vector}: This figure shows the temperature field and the heat flux vectors for the three vascular layouts; the white arrows represent the heat flux vectors. The first column collates the results under CMP, while the results in the second column are for TDMP. There are no appreciable differences in the solution fields between the CMP and TDMP cases.
    \label{Fig7:TProp_Temp_profile}}
\end{figure}

\textbf{Figures~\ref{Fig8:U_shape_thermal_efficiency}} to \textbf{\ref{Fig10:Asymmetry_thermal_efficiency}} depict the time-dependent changes in thermal efficiencies for the three vascular configurations. These visual representations also zoom on the central segments to highlight the magnitude of fluctuations. It is evident from these graphs that, during the transient phase, there are slight distinctions between the CMP and TDMP scenarios. But, as time progresses towards a steady state, the disparities between the two scenarios decrease, a pattern observed across all three materials.

\begin{figure}[h!]
    \centering
    \includegraphics[scale=0.28]{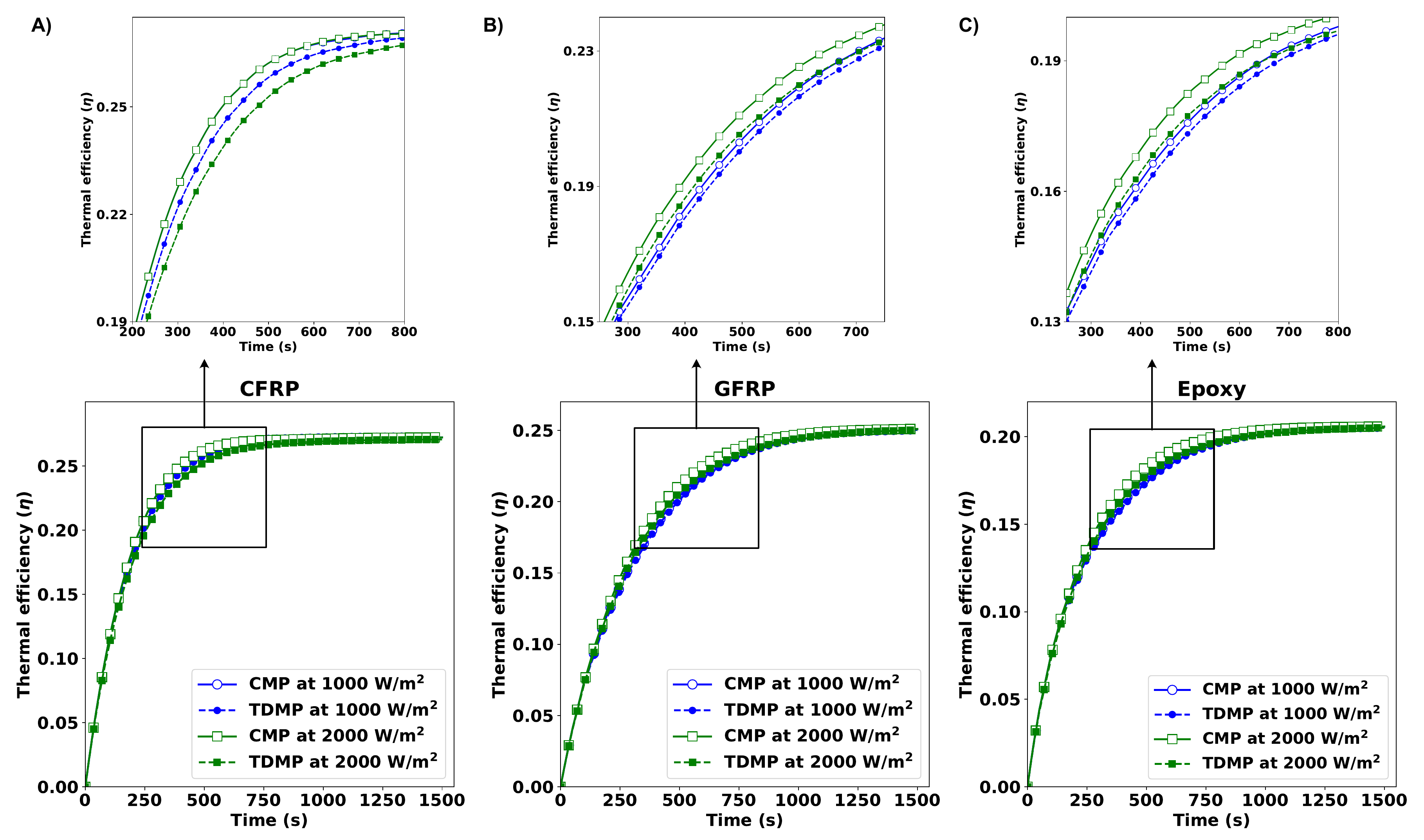}
    \caption{\textsf{Thermal efficiency --- U-shaped vasculature.}  For all three materials in the transient regime, the thermal efficiency exhibits minor disparities between constant material properties (CMP) and temperature-dependent properties (TDMP). However, this disparity in thermal efficiency diminishes as the system approaches a steady state. This slight deviation in the intermediate phase is magnified to observe the nature of the discrepancy more closely. In each scenario, the variation remains minimal. Also, the thermal efficiency remains nearly indistinguishable for both heat fluxes $1000 \; \mathrm{W/m^2}$ and $2000 \; \mathrm{W/m^2}$. \label{Fig8:U_shape_thermal_efficiency}}
\end{figure}

\begin{figure}[h]
    \centering
    \includegraphics[scale=0.28]{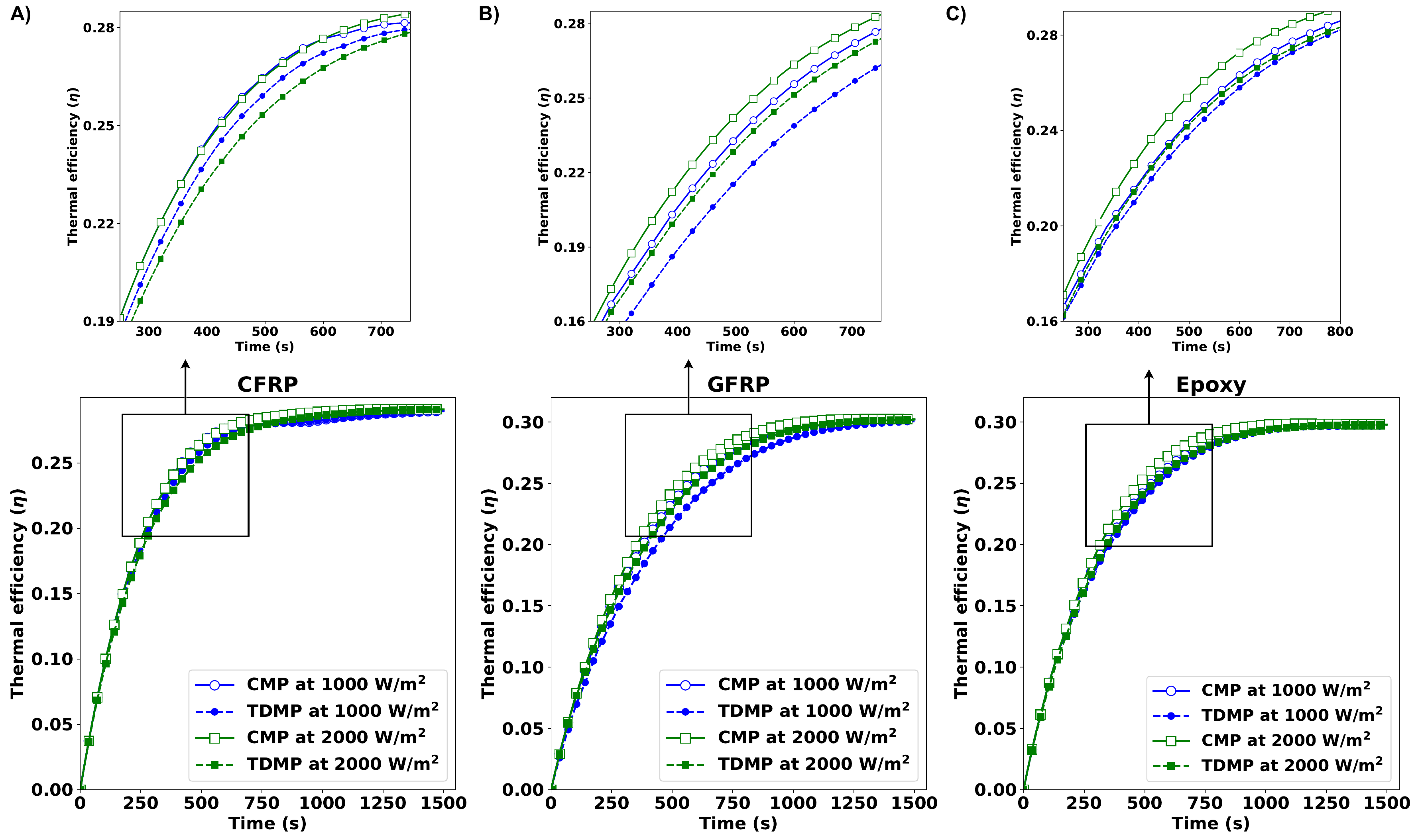}
    \caption{\textsf{Thermal efficiency --- Serpentine vasculature}. This figure compares the thermal efficiency obtained under constant \emph{vs.} temperature-dependent material properties for three materials: \textbf{A)} CFRP, \textbf{B)} GFRP, and \textbf{C)} epoxy. Two different heat fluxes are considered: 1000 $\mathrm{W/m^2}$ and 2000 $\mathrm{W/m^2}$. The top panel shows a zoomed view of the temperature evolution in the transient regime. The thermal efficiency of the materials exhibits a minor disparity in the transient regime. However, the difference fades away as the system approaches a steady state over time. \label{Fig9:Serpentine_thermal_efficiency}}
\end{figure}

\begin{figure}[!h]
    \centering
    \includegraphics[scale=0.28]{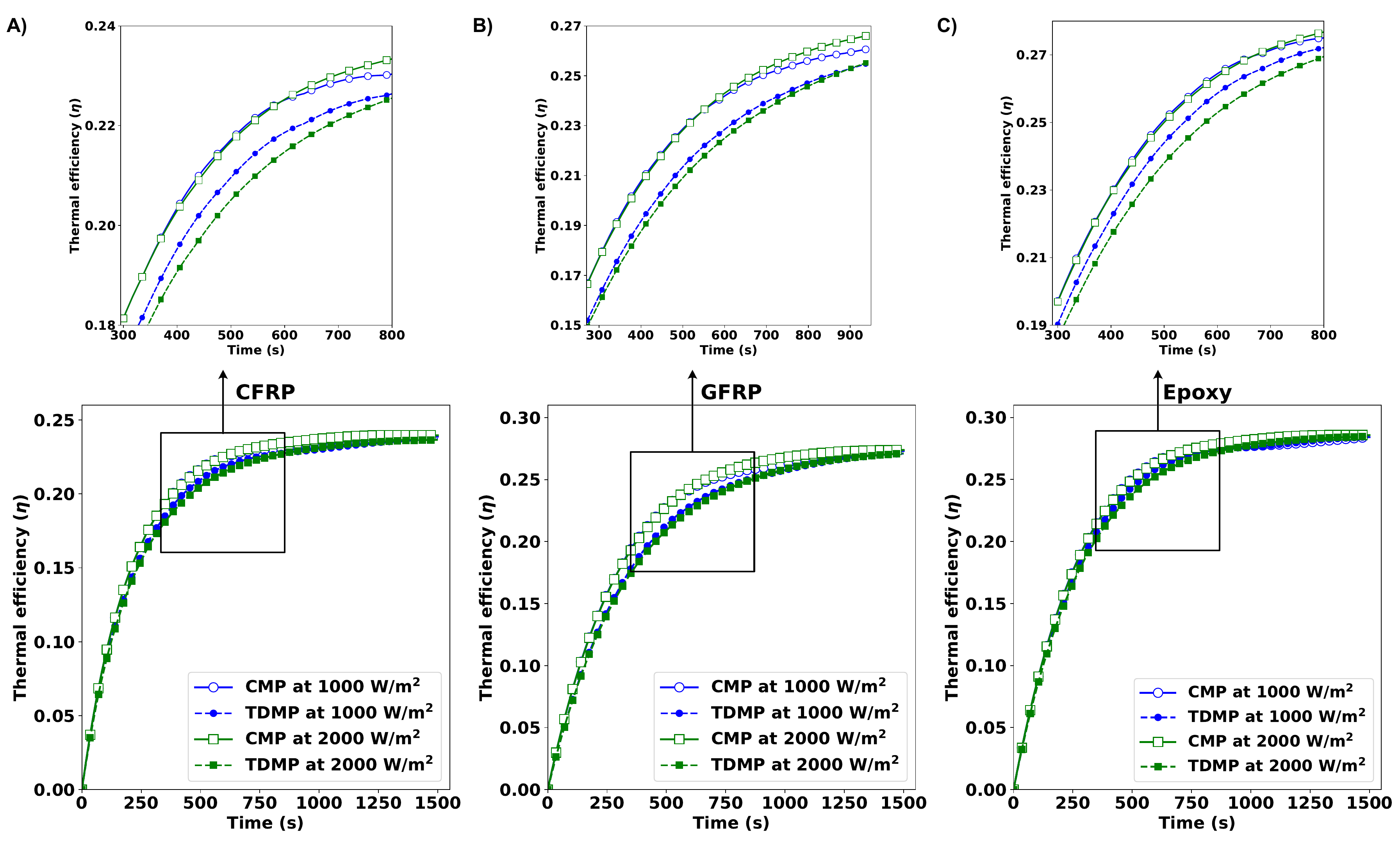}
    \caption{\textsf{Thermal efficiency --- Asymmetric vasculature.} The thermal efficiency of the materials exhibits more prominent disparity in the transient regime as compared to U-shaped and serpentine vasculatures. However, as the system approaches a steady state over time, the efficiencies converge to a uniform value. The deviation in the intermediate phase is magnified to observe the nature of the discrepancy more closely. In each scenario, the variations are negligible. Also, the thermal efficiency remains nearly indistinguishable for both heat fluxes $1000 \; \mathrm{W/m^2}$ and $2000 \; \mathrm{W/m^2}$.\label{Fig10:Asymmetry_thermal_efficiency}}
\end{figure}
\textbf{Figure~\ref{Fig11:Arclength_temperature}} plots the temperature along the arc-length of the U-shaped vasculature. For all three materials (CFRP, GFRP, and epoxy) and both heat fluxes $1000 \; \mathrm{W/m^2}$ and $2000 \; \mathrm{W/m^2}$, the temperature profiles along the vasculature compare well under CMP and TDMP cases. However, temperature differences are relatively more pronounced for CFRP; nonetheless, these differences become almost negligible towards the outlet.

For the considered temperature range and for the realistic temperature-dependence of the material properties, there are no significant quantitative differences compared to the response under constant material properties---answering the third question (Q3). 
\begin{figure}[h]
    \centering
    \includegraphics[scale=0.35]{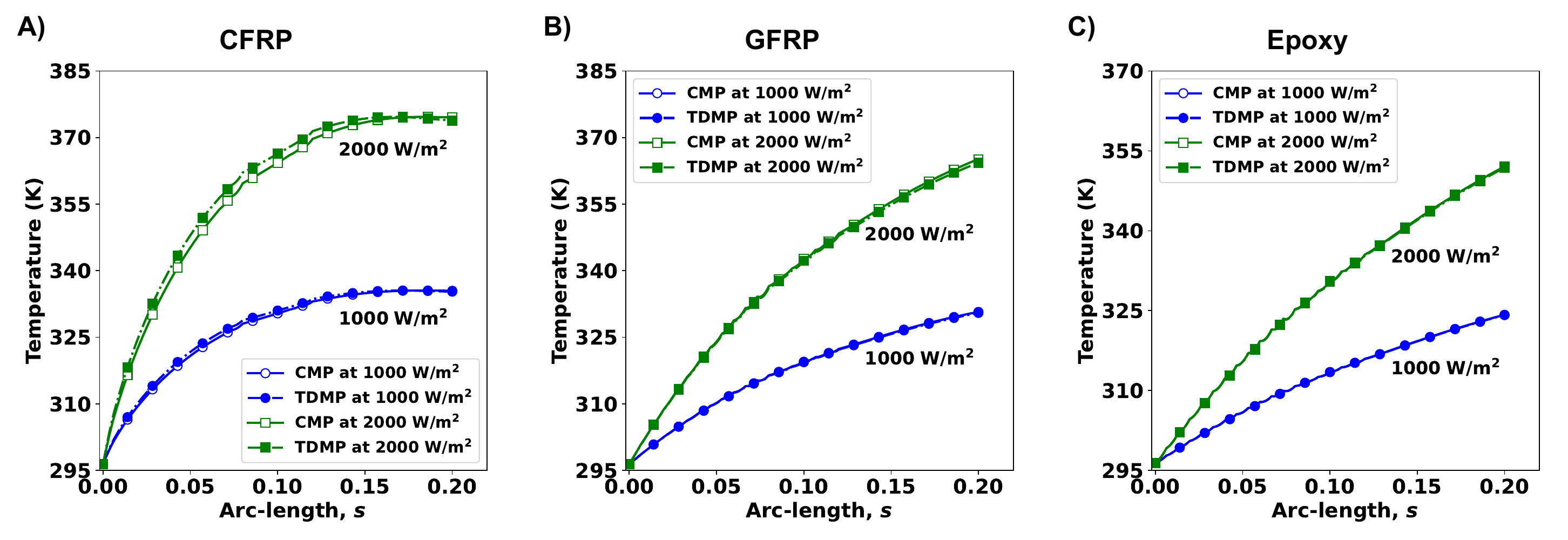}
    \caption{\textsf{Arc-length temperature.} The arc-length temperatures for the three materials are graphed considering U-shaped vasculature under two different heat fluxes: $1000$ and $2000 \; \mathrm{W/m^2}$. Arc-length is calculated as the distance measured from the inlet while moving along the vasculature. It is apparent from the charts that the influence of temperature-dependent material properties (TDMP) on arc-length temperature is not significant compared to the constant material properties (CMP) scenario. \label{Fig11:Arclength_temperature}}
\end{figure}

\section{CLOSURE}
\label{Sec:S8_TProp_Closure}
This paper considered three materials: carbon-fiber-reinforced polymer (CFRP), glass-fiber-reinforced polymer (GFRP), and epoxy. Experimental characterization has revealed that the mentioned materials' thermophysical properties (i.e., specific heat capacity and thermal conductivity) depend on the temperature. This study revealed the ramifications of the remarked dependence on vascular-based active cooling performance. 

The main conclusions are: 
\begin{enumerate}[(C1)]
    \item In the chosen temperature range 23--150 $^\circ\mathrm{C}$, the specific heat capacities of the three materials depend on the temperature, with epoxy showing stronger dependence compared to CFRP and GFRP. On the other hand, the thermal conductivity of CFRP increases mildly with temperature for the same range (approximately with a slope of 0.01), while the thermal conductivities of GFRP and epoxy do not vary appreciably with the temperature.
    \item For typical flow rates encountered in active cooling of microvascular composites (i.e., order of mL/min) and for the temperature range 23--100 $^\circ\mathrm{C}$, the heat capacity rate of water---the coolant---does not vary with temperature. 
    \item The solution field under the mathematical model accounting for temperature-dependent thermophysical properties also possesses minimum and maximum principles under steady state, mirroring the recent findings for constant material properties \citep{nakshatrala2023ROM}. However, such principles are yet to be explored for solutions in the transient regime, even when material properties are constant.
    \item The temperature dependence of material properties exhibited by the three materials does not affect the two invariants---mean surface temperature and outlet temperature---under flow reversal, extending the previously known results for constant material properties. Thus, these two invariants under flow reversal could still serve as a suitable measure for studying thermal regulation using such materials in the mentioned temperature range.
    \item The temperature-dependent material properties do not appreciably affect---either qualitatively or quantitatively---the active cooling performance. Specifically, this dependence has minimal effect on the global thermal characteristics such as the mean surface temperature and thermal efficiency.
\end{enumerate}

A plausible future work can be towards studying thermo-mechanical coupled response in active-cooling applications, conjointly with the dependence of material properties on temperature.  

\section*{DATA AVAILABILITY}
The article includes all the study data (temporal evolution of temperature, characterization techniques, methods, material parameters, extended investigation results, and modeling details). Any other data supporting this study's findings are available from the corresponding author (KBN) upon reasonable request.

\bibliographystyle{plainnat}
\bibliography{Master_References}

\begin{thebibliography}{51}
\providecommand{\natexlab}[1]{#1}
\providecommand{\url}[1]{\texttt{#1}}
\expandafter\ifx\csname urlstyle\endcsname\relax
  \providecommand{\doi}[1]{doi: #1}\else
  \providecommand{\doi}{doi: \begingroup \urlstyle{rm}\Url}\fi

\bibitem[Ashbee(1993)]{ashbee1993fundamental}
K.~H.~G. Ashbee.
\newblock \emph{{Fundamental Principles of Fiber Reinforced Composites}}.
\newblock CRC Press, Boca Raton, 1993.

\bibitem[Bandhauer et~al.(2011)Bandhauer, Garimella, and
  Fuller]{bandhauer2011critical}
T.~M. Bandhauer, S.~Garimella, and T.~F. Fuller.
\newblock A critical review of thermal issues in lithium-ion batteries.
\newblock \emph{Journal of the Electrochemical Society}, 158\penalty0
  (3):\penalty0 R1, 2011.
\newblock \doi{10.1149/1.3515880/meta}.

\bibitem[Benfeldt(1999)]{benfeldt1999vivo}
E.~Benfeldt.
\newblock {In vivo microdialysis for the investigation of drug levels in the
  dermis and the effect of barrier perturbation on cutaneous drug penetration.
  Studies in hairless rats and human subjects}.
\newblock \emph{Acta dermato-venereologica. Supplementum}, 206:\penalty0 1--59,
  1999.

\bibitem[Br{\'e}zis(2011)]{brezis2011functional}
H.~Br{\'e}zis.
\newblock \emph{{Functional Analysis, Sobolev Spaces and Partial Differential
  Equations}}.
\newblock Springer, New York, 2011.

\bibitem[Buck and Rudtsch(2011)]{buck2011thermal}
W.~Buck and S.~Rudtsch.
\newblock Thermal properties.
\newblock \emph{Springer Handbook of Metrology and Testing}, pages 453--483,
  2011.
\newblock \doi{10.1007/978-3-642-16641-9_8}.

\bibitem[\c{C}engel et~al.(2011)\c{C}engel, Boles, and
  Kanoglu]{cengel2011thermodynamics}
Y.~A. \c{C}engel, M.~A. Boles, and M.~Kanoglu.
\newblock \emph{{Thermodynamics: An Engineering Approach}}.
\newblock McGraw-Hill, New York, 2011.

\bibitem[{COMSOL Multiphysics}(2018)]{COMSOL}
{COMSOL Multiphysics}.
\newblock \emph{{Comsol User's Guide, Version 5.3}}.
\newblock COMSOL AB, Stockholm, Sweden, 2018.

\bibitem[Cui et~al.(2012)Cui, Gao, and Zhang]{cui2012new}
M.~Cui, X.~Gao, and J.~Zhang.
\newblock A new approach for the estimation of temperature-dependent thermal
  properties by solving transient inverse heat conduction problems.
\newblock \emph{International Journal of Thermal Sciences}, 58:\penalty0
  113--119, 2012.
\newblock \doi{10.1016/j.ijthermalsci.2012.02.024}.

\bibitem[de~Faoite et~al.(2012)de~Faoite, Browne, Chang-D{\'\i}az, and
  Stanton]{de2012review}
D.~de~Faoite, D.~J. Browne, F.~R. Chang-D{\'\i}az, and K.~T. Stanton.
\newblock A review of the processing, composition, and temperature-dependent
  mechanical and thermal properties of dielectric technical ceramics.
\newblock \emph{Journal of Materials Science}, 47:\penalty0 4211--4235, 2012.
\newblock \doi{10.1007/s10853-011-6140-1}.

\bibitem[Devi et~al.(2021)Devi, Pejman, Phillips, Zhang, Soghrati, Nakshatrala,
  Najafi, Schab, and Patrick]{devi2021microvascular}
U.~Devi, R.~Pejman, Z.~J. Phillips, P.~Zhang, S.~Soghrati, K.~B. Nakshatrala,
  A.~R. Najafi, K.~R. Schab, and J.~F. Patrick.
\newblock A microvascular-based multifunctional and reconfigurable
  metamaterial.
\newblock \emph{Advanced Materials Technologies}, 6\penalty0 (11):\penalty0
  2100433, 2021.
\newblock \doi{10.1002/admt.202100433}.

\bibitem[Devi et~al.(2023)Devi, Kumar, Nakshatrala, and
  Patrick]{devi2023methodology}
U.~Devi, S.~R. Kumar, K.~B. Nakshatrala, and J.~F. Patrick.
\newblock A methodology for measuring heat transfer coefficient and
  self-similarity of thermal regulation in microvascular material systems.
\newblock \emph{International Journal of Heat and Mass Transfer}, 217:\penalty0
  124614, 2023.
\newblock \doi{10.1016/j.ijheatmasstransfer.2023.124614}.

\bibitem[Driesma et~al.(2019)Driesma, Ercol, Gaddy, and
  Gerger]{driesman2019journey}
A.~Driesma, J.~Ercol, E.~Gaddy, and A.~Gerger.
\newblock {Journey to the center of the solar system: How the Parker solar
  probe survives close encounters with the sun}.
\newblock \emph{IEEE Spectrum}, 56\penalty0 (5):\penalty0 32--53, 2019.
\newblock \doi{10.1109/MSPEC.2019.8701197}.

\bibitem[engineering toolbox()]{engineeringtoolbox}
The engineering toolbox.
\newblock Water -- specific heat vs. temperature.
\newblock URL
  \url{https://www.engineeringtoolbox.com/specific-heat-capacity-water-d_660.html}.
\newblock Accessed on 10-13-2023.

\bibitem[Favero et~al.(2021)Favero, Bonesso, Rebesan, Dima, Pepato, and
  Mancin]{favero2021additive}
G.~Favero, M.~Bonesso, P.~Rebesan, R.~Dima, A.~Pepato, and S.~Mancin.
\newblock Additive manufacturing for thermal management applications: from
  experimental results to numerical modeling.
\newblock \emph{International Journal of Thermofluids}, 10:\penalty0 100091,
  2021.
\newblock \doi{10.1016/j.ijft.2021.100091}.

\bibitem[Gilbarg and Trudinger(2015)]{gilbarg2015elliptic}
D.~Gilbarg and N.~S. Trudinger.
\newblock \emph{{Elliptic Partial Differential Equations of Second Order}}.
\newblock Springer-Verlag, New York, 2015.

\bibitem[Godovsky(2012)]{godovsky2012thermophysical}
Y.~K. Godovsky.
\newblock \emph{{Thermophysical Properties of Polymers}}.
\newblock Springer-Verlag, Berlin, 2012.

\bibitem[Gonz{\'a}lez-Alonso(2012)]{gonzalez2012human}
J.~Gonz{\'a}lez-Alonso.
\newblock Human thermoregulation and the cardiovascular system.
\newblock \emph{Experimental Physiology}, 97\penalty0 (3):\penalty0 340--346,
  2012.
\newblock \doi{10.1113/expphysiol.2011.058701}.

\bibitem[Grimvall(1999)]{grimvall1999thermophysical}
G.~Grimvall.
\newblock \emph{{Thermophysical Properties of Materials}}.
\newblock Elsevier, Amsterdam, 1999.

\bibitem[Hill and Veghte(1976)]{hill1976jackrabbit}
R.~W. Hill and J.~H. Veghte.
\newblock {Jackrabbit ears: Surface temperatures and vascular responses}.
\newblock \emph{Science}, 194\penalty0 (4263):\penalty0 436--438, 1976.
\newblock \doi{10.1126/science.982027}.

\bibitem[Huang and Jan-Yuan(1995)]{huang1995inverse}
C.~H. Huang and Y.~Jan-Yuan.
\newblock An inverse problem in simultaneously measuring temperature-dependent
  thermal conductivity and heat capacity.
\newblock \emph{International Journal of Heat and Mass Transfer}, 38\penalty0
  (18):\penalty0 3433--3441, 1995.
\newblock \doi{10.1016/0017-9310(95)00059-I}.

\bibitem[Hyer and White(2009)]{hyer2009stress}
M.~W. Hyer and S.~R. White.
\newblock \emph{{Stress Analysis of Fiber-Reinforced Composite Materials}}.
\newblock DEStech Publications, Inc, Lancaster, Pennsylvania, 2009.

\bibitem[Jagtap et~al.(2023)Jagtap, Mudunuru, and
  Nakshatrala]{jagtap2023coolpinns}
N.~V. Jagtap, M.~K. Mudunuru, and K.~B. Nakshatrala.
\newblock {CoolPINNs: A physics-informed neural network modeling of active
  cooling in vascular systems}.
\newblock \emph{Applied Mathematical Modelling}, 122:\penalty0 265--287, 2023.
\newblock \doi{10.1016/j.apm.2023.04.020}.

\bibitem[Kahl(1963)]{kahl1963thermoregulation}
M.~P. Kahl.
\newblock Thermoregulation in the wood stork, with special reference to the
  role of the legs.
\newblock \emph{Physiological Zoology}, 36\penalty0 (2):\penalty0 141--151,
  1963.
\newblock \doi{10.1086/physzool.36.2.30155437}.

\bibitem[Kiomarsipour et~al.(2013)Kiomarsipour, Razavi, and
  Ghani]{kiomarsipour2013improvement}
N.~Kiomarsipour, R.~S. Razavi, and K.~Ghani.
\newblock {Improvement of spacecraft white thermal control coatings using the
  new synthesized Zn-MCM-41 pigment}.
\newblock \emph{Dyes and Pigments}, 96\penalty0 (2):\penalty0 403--406, 2013.
\newblock \doi{10.1016/j.dyepig.2012.08.019}.

\bibitem[Kittel and Kroemer(1998)]{kittel1998thermal}
C.~Kittel and H.~Kroemer.
\newblock \emph{{Thermal Physics}}.
\newblock W.~H.~Freeman and Company, New York, second edition, 1998.

\bibitem[Mahamud and Park(2011)]{mahamud2011reciprocating}
R.~Mahamud and C.~Park.
\newblock {Reciprocating air flow for Li-ion battery thermal management to
  improve temperature uniformity}.
\newblock \emph{Journal of Power Sources}, 196\penalty0 (13):\penalty0
  5685--5696, 2011.
\newblock \doi{10.1016/j.jpowsour.2011.02.076}.

\bibitem[Nakshatrala(2023)]{nakshatrala2023ROM}
K.~B. Nakshatrala.
\newblock {Modeling thermal regulation in thin vascular systems: A mathematical
  analysis}.
\newblock \emph{Communications in Computational Physics}, 33:\penalty0
  1035--1068, 2023.
\newblock \doi{10.4208/cicp.OA-2022-0240}.

\bibitem[Nakshatrala and Adhikari(2024)]{nakshatrala2023thermal}
K.~B. Nakshatrala and K.~Adhikari.
\newblock {Thermal regulation in thin vascular systems: A sensitivity
  analysis}.
\newblock \emph{Communication in Computational Physics}, 35\penalty0
  (2):\penalty0 427--466, 2024.
\newblock \doi{10.4208/cicp.OA-2023-0166}.

\bibitem[Nakshatrala et~al.(2023)Nakshatrala, Adhikari, Kumar, and
  Patrick]{Nakshatrala_PNAS_Nexus_2023}
K.~B. Nakshatrala, K.~Adhikari, S.~R. Kumar, and J.~F. Patrick.
\newblock Configuration-independent thermal invariants under flow reversal in
  thin vascular systems.
\newblock \emph{PNAS Nexus}, 2\penalty0 (8):\penalty0 pgad266, 2023.
\newblock \doi{10.1093/pnasnexus/pgad266}.

\bibitem[Nguyen et~al.(2018)Nguyen, Park, Zhang, and Liang]{nguyen2018recent}
N.~Nguyen, J.~G. Park, S.~Zhang, and R.~Liang.
\newblock {Recent advances on 3D printing technique for thermal-related
  applications}.
\newblock \emph{Advanced Engineering Materials}, 20\penalty0 (5):\penalty0
  1700876, 2018.
\newblock \doi{10.1002/adem.201700876}.

\bibitem[Nitta et~al.(2015)Nitta, Wu, Lee, and Yushin]{nitta2015li}
N.~Nitta, F.~Wu, J.~T. Lee, and G.~Yushin.
\newblock Li-ion battery materials: present and future.
\newblock \emph{Materials Today}, 18\penalty0 (5):\penalty0 252--264, 2015.
\newblock \doi{10.1016/j.mattod.2014.10.040}.

\bibitem[Ordonez-Miranda et~al.(2018)Ordonez-Miranda, Ezzahri, Joulain,
  Drevillon, and Alvarado-Gil]{ordonez2018modeling}
J.~Ordonez-Miranda, Y.~Ezzahri, K.~Joulain, J.~Drevillon, and J.~J.
  Alvarado-Gil.
\newblock {Modeling of the electrical conductivity, thermal conductivity, and
  specific heat capacity of $\mathrm{VO_2}$}.
\newblock \emph{Physical Review B}, 98\penalty0 (7):\penalty0 075144, 2018.
\newblock \doi{10.1103/PhysRevB.98.075144}.

\bibitem[Oueslati et~al.(2008)Oueslati, Therriault, and
  Martel]{oueslati2008pcb}
R.~B. Oueslati, D.~Therriault, and S.~Martel.
\newblock {PCB-integrated heat exchanger for cooling electronics using
  microchannels fabricated with the direct-write method}.
\newblock \emph{IEEE Transactions on Components and Packaging Technologies},
  31\penalty0 (4):\penalty0 869--874, 2008.
\newblock \doi{10.1109/TCAPT.2008.2004773}.

\bibitem[Pao(2012)]{pao2012nonlinear}
C-V. Pao.
\newblock \emph{{Nonlinear Parabolic and Elliptic Equations}}.
\newblock Springer Science \& Business Media, New York, 2012.

\bibitem[Pastukhov et~al.(2003)Pastukhov, Maidanik, Vershinin, and
  Korukov]{pastukhov2003miniature}
V.~G. Pastukhov, Y.~F. Maidanik, C.~V. Vershinin, and M.~A. Korukov.
\newblock Miniature loop heat pipes for electronics cooling.
\newblock \emph{Applied Thermal Engineering}, 23\penalty0 (9):\penalty0
  1125--1135, 2003.
\newblock \doi{10.1016/S1359-4311(03)00046-2}.

\bibitem[Pathak et~al.(2023)Pathak, Majkic, Erickson, Chen, and
  Selvamanickam]{pathak4529173two}
P.~Pathak, G.~Majkic, T.~Erickson, T.~Chen, and V.~Selvamanickam.
\newblock {Two-dimensional X-ray diffraction (2D-Xrd) and micro-computed
  tomography (micro-Ct) characterization of additively manufactured 316l
  stainless steel}.
\newblock \emph{Preprint available at SSRN 4529173}, 2023.
\newblock \doi{10.2139/ssrn.4529173}.

\bibitem[Patrick et~al.(2017)Patrick, Krull, Garg, Mangun, Moore, Sottos, and
  White]{patrick2017robust}
J.~F. Patrick, B.~P. Krull, M.~Garg, C.~L. Mangun, J.~S. Moore, N.~R. Sottos,
  and S.~R. White.
\newblock Robust sacrificial polymer templates for 3d interconnected
  microvasculature in fiber-reinforced composites.
\newblock \emph{Composites Part A: Applied Science and Manufacturing},
  100:\penalty0 361--370, 2017.
\newblock \doi{10.1016/j.compositesa.2017.05.022}.

\bibitem[Pety et~al.(2017)Pety, Tan, Najafi, Barnett, Geubelle, and
  White]{pety2017carbon}
S.~J. Pety, M.~H.~Y. Tan, A.~R. Najafi, P.~R. Barnett, P.~H. Geubelle, and
  S.~R. White.
\newblock {Carbon fiber composites with 2D microvascular networks for battery
  cooling}.
\newblock \emph{International Journal of Heat and Mass Transfer}, 115:\penalty0
  513--522, 2017.
\newblock \doi{10.1016/j.ijheatmasstransfer.2017.07.047}.

\bibitem[Singh(2006)]{singh2006thermal}
R.~Singh.
\newblock \emph{Thermal control of high-powered desktop and laptop
  microprocessors using two-phase and single-phase loop cooling systems}.
\newblock PhD thesis, RMIT University, 2006.

\bibitem[Sklan and Li(2018)]{sklan2018thermal}
S.~R. Sklan and B.~Li.
\newblock Thermal metamaterials: functions and prospects.
\newblock \emph{National Science Review}, 5\penalty0 (2):\penalty0 138--141,
  2018.
\newblock \doi{10.1093/nsr/nwy005}.

\bibitem[Spencer et~al.(2007)Spencer, Joseph, Kim, Swaminathan, and
  Kohl]{spencer2007air}
T.~J. Spencer, P.~J. Joseph, T.~H. Kim, M.~Swaminathan, and P.~A. Kohl.
\newblock Air-gap transmission lines on organic substrates for low-loss
  interconnects.
\newblock \emph{IEEE Transactions on Microwave Theory and Techniques},
  55\penalty0 (9):\penalty0 1919--1925, 2007.
\newblock \doi{10.1109/TMTT.2007.904326}.

\bibitem[Steen and Steen(1965)]{steen1965importance}
I.~Steen and J.~B. Steen.
\newblock The importance of the legs in the thermoregulation of birds.
\newblock \emph{Acta Physiologica Scandinavica}, 63\penalty0 (3):\penalty0
  285--291, 1965.
\newblock \doi{10.1111/j.1748-1716.1965.tb04067.x}.

\bibitem[Swanson and Birur(2003)]{swanson2003nasa}
T.~D. Swanson and G.~C. Birur.
\newblock {NASA thermal control technologies for robotic spacecraft}.
\newblock \emph{Applied Thermal Engineering}, 23\penalty0 (9):\penalty0
  1055--1065, 2003.
\newblock \doi{10.1016/S1359-4311(03)00036-X}.

\bibitem[Tan et~al.(2018)Tan, Bunce, Ghosh, and Geubelle]{tan2018computational}
M.~H.~Y. Tan, D.~Bunce, A.~R.~M. Ghosh, and P.~H. Geubelle.
\newblock Computational design of microvascular radiative cooling panels for
  nanosatellites.
\newblock \emph{Journal of Thermophysics and Heat Transfer}, 32\penalty0
  (3):\penalty0 605--616, 2018.
\newblock \doi{10.2514/1.T5381}.

\bibitem[{TPRL}(2023)]{TPRL}
{TPRL}.
\newblock \emph{{Thermophysical Properties Research Laboratory, Inc.}}
\newblock 3080 Kent Avenue, West Lafayette, IN 47906, 2023.
\newblock \url{http://www.tprl.com}.

\bibitem[Tritt(2004)]{tritt2005thermal}
T.~M. Tritt.
\newblock \emph{{Thermal Conductivity: Theory, Properties, and Applications}}.
\newblock Plenum Publishers, New York, 2004.

\bibitem[USGS(2018)]{scienceschool}
USGS.
\newblock Water's density varies with temperature, 2018.
\newblock URL
  \url{https://www.usgs.gov/special-topics/water-science-school/science/water-density}.
\newblock Accessed on 10-13-2023.

\bibitem[Wen et~al.(2015)Wen, Lu, Xiao, and Deng]{wen2015temperature}
H.~Wen, J.~H. Lu, Y.~Xiao, and J.~Deng.
\newblock Temperature dependence of thermal conductivity, diffusion and
  specific heat capacity for coal and rocks from coalfield.
\newblock \emph{Thermochimica Acta}, 619:\penalty0 41--47, 2015.
\newblock \doi{10.1016/j.tca.2015.09.018}.

\bibitem[Wissler(2018)]{wissler2018animal}
E.~H. Wissler.
\newblock Animal heat and thermal regulation.
\newblock \emph{Human Temperature Control: A Quantitative Approach}, pages
  1--16, 2018.
\newblock \doi{10.1007/978-3-662-57397-6_1}.

\bibitem[Yang et~al.(2019)Yang, Foulkes, Kwon, Kang, Braun, and and
  Miljkovic]{yang2019integrated}
T.~Yang, T.~Foulkes, B.~Kwon, J.~G. Kang, P.~V. Braun, and W.~P.~King~N. and
  Miljkovic.
\newblock An integrated liquid metal thermal switch for active thermal
  management of electronics.
\newblock \emph{IEEE Transactions on Components, Packaging and Manufacturing
  Technology}, 9\penalty0 (12):\penalty0 2341--2351, 2019.
\newblock \doi{10.1109/TCPMT.2019.2930089}.

\bibitem[Yu et~al.(2007)Yu, Till, and Thomas]{yu2007modeling}
B.~Yu, V.~Till, and K.~Thomas.
\newblock {Modeling of thermo-physical properties for FRP composites under
  elevated and high temperature}.
\newblock \emph{Composites Science and Technology}, 67\penalty0
  (15-16):\penalty0 3098--3109, 2007.
\newblock \doi{10.1016/j.compscitech.2007.04.019}.

\end{thebibliography}
\end{document}